\documentclass[11pt]{article}
\usepackage[utf8]{inputenc}
\usepackage{amsmath,amsthm,amsfonts}
\usepackage[textheight=8.5in,textwidth=6.5in]{geometry}
\usepackage{mathrsfs}
\usepackage{graphicx}
\usepackage{hyperref}
\usepackage{makecell}
\usepackage{thm-restate}
\usepackage{booktabs} 

\allowdisplaybreaks

\usepackage{fancybox, pdfsync}

\usepackage[noend,linesnumbered]{algorithm2e}

\newenvironment{fminipage}%
  {\begin{Sbox}\begin{minipage}}%
  {\end{minipage}\end{Sbox}\fbox{\TheSbox}}

\usepackage{bbm}

\newtheorem{theorem}{Theorem}

\newtheorem{lemma}[theorem]{Lemma}
\newtheorem{fact}[theorem]{Fact}

\newtheorem{conjecture}[theorem]{Conjecture}
\newtheorem{definition}{Definition}

\newtheorem*{remark*}{Remark}

\usepackage{color}

\def\nnz#1{\mbox{\textsc{nnz}}\left( #1 \right)}
\newcommand{\real}{\text{real}}
\newcommand{\imag}{\text{imag}}
\newcommand{\rank}{\text{rank}}
\newcommand{\matmult}{\text{MM}}
\def\poly#1{\mbox{poly}\left( #1 \right)}

\def\abs#1{\left|#1  \right|}

\def\norm#1{\left\| #1 \right\|}

\newcommand\R{\mathbb{R}}
\newcommand\F{\mathbb{F}}

\newcommand\Otil{\tilde{O}}

\usepackage{enumitem}

\newcommand{\va}{\mathbf{a}}
\newcommand{\vb}{\mathbf{b}}
\newcommand{\vc}{\mathbf{c}}
\newcommand{\vx}{\mathbf{x}}
\newcommand{\vy}{\mathbf{y}}
\newcommand{\vg}{\mathbf{g}}
\newcommand{\vecv}{\mathbf{v}}
\newcommand{\vecvbar}{\widetilde{\mathbf{v}}}
\newcommand{\vv}{\mathbf{v}}
\newcommand{\vh}{\mathbf{h}}
\newcommand{\vw}{\mathbf{w}}
\newcommand{\vwtil}{\widetilde{\mathbf{w}}}
\newcommand{\ma}{\mathbf{A}}
\newcommand{\mb}{\mathbf{B}}
\newcommand{\me}{\mathbf{E}}
\newcommand{\metil}{\widetilde{\mathbf{E}}}
\newcommand{\mf}{\mathbf{F}}
\newcommand{\mh}{\mathbf{H}}
\newcommand{\mg}{\mathbf{G}}
\newcommand{\mk}{\mathbf{K}}
\newcommand{\mc}{\mathbf{C}}
\newcommand{\md}{\mathbf{D}}
\newcommand{\mdtil}{\widetilde{\mathbf{D}}}
\newcommand{\ml}{\mathbf{L}}
\newcommand{\mm}{\mathbf{M}}
\newcommand{\mw}{\mathbf{W}}
\newcommand{\mwtil}{\widetilde{\mathbf{W}}}
\newcommand{\matu}{\mathbf{U}}
\newcommand{\mn}{\mathbf{N}}
\newcommand{\mi}{\mathbf{I}}
\newcommand{\mj}{\mathbf{J}}

\newcommand{\ms}{\mathbf{S}}
\newcommand{\mt}{\mathbf{T}}
\renewcommand{\mp}{\mathbf{P}}
\newcommand{\mq}{\mathbf{Q}}
\newcommand{\mx}{\mathbf{X}}
\newcommand{\mv}{\mathbf{V}}
\newcommand{\my}{\mathbf{Y}}
\newcommand{\mz}{\mathbf{Z}}

\newcommand{\mDelta}{\mathbf{\Delta}}
\newcommand{\mDeltatil}{\widetilde{\mathbf{\Delta}}}

\newcommand{\mbtil}{\widetilde{\mathbf{B}}}
\newcommand{\mctil}{\widetilde{\mathbf{C}}}
\newcommand{\mstil}{\widetilde{\mathbf{S}}}
\newcommand{\mxtil}{\widetilde{\mathbf{X}}}
\newcommand{\mytil}{\widetilde{\mathbf{Y}}}
\newcommand{\mhtil}{\widetilde{\mathbf{H}}}
\newcommand{\mptil}{\widetilde{\mathbf{P}}}
\newcommand{\mttil}{\widetilde{\mathbf{T}}}
\newcommand{\vvtil}{\widetilde{\mathbf{v}}}
\newcommand{\vhtil}{\widetilde{\mathbf{h}}}

\newcommand{\C}{\mathbb{C}}

\newcommand{\N}{\mathbb{N}}

\newcommand{\fro}{\textnormal{F}}

\date{}
\begin{document}

\title{
On Symmetric Factorizations of Hankel Matrices
}
\author{Mehrdad Ghadiri\footnote{Georgia Institute of Technology, \url{ghadiri@gatech.edu}} }

\maketitle
\thispagestyle{empty}
\begin{abstract}
We present two conjectures regarding the running time of computing symmetric factorizations for a Hankel matrix $\mathbf{H}$ and its inverse $\mathbf{H}^{-1}$ as $\mathbf{B}\mathbf{B}^*$ under fixed-point arithmetic. If solved, these would result in a faster-than-matrix-multiplication algorithm for solving sparse poly-conditioned linear programming problems, a fundamental problem in optimization and theoretical computer science. To justify our proposed conjectures and running times, we show weaker results of computing decompositions of the form $\mathbf{B}\mathbf{B}^* - \mathbf{C}\mathbf{C}^*$ for Hankel matrices and their inverses with the same running time. In addition, to promote our conjectures further, we discuss the connections of Hankel matrices and their symmetric factorizations to sum-of-squares (SoS) decompositions of single-variable polynomials.
\end{abstract}
\newpage

\section{Introduction}

Linear system solvers are a workhorse of the modern approach to optimization in which a linear system is solved in each iteration. This approach has been adapted for many problems ranging from graph problems \cite{christiano2011electrical}, to $p$-norm regression \cite{adil2019iterative}, and linear programming \cite{Karmarkar84,vaidya1989speeding,cohen2021solving}. If the linear systems in the problem have a special \emph{structure}, then the structure can usually be exploited to obtain faster algorithms. This has probably been best exemplified by near-linear time Laplacian solvers that have led to improved running times in many graph problems \cite{spielman2014nearly,daitch2008faster,koutis2014approaching,gao2023robust}.

Solving a general linear system and various factorization of matrices can be done in $O(n^3)$ arithmetic (or field) operations. This can be improved using \emph{fast matrix multiplication} techniques to $O(n^\omega)$, where $\omega<2.373$ is the matrix multiplication exponent \cite{alman2021refined,coppersmith1982asymptotic,strassen1969gaussian}. For solving linear systems with structured matrices such as Hankel and Toeplitz, \emph{fast} algorithms with $O(n^2)$ arithmetic operations have been presented \cite{bareiss1969numerical,chun1989fast}. This can be improved to algorithms with $\Otil(n)$ arithmetic operations. These are called \emph{super fast solvers} \cite{ammar1988superfast,van1998stabilized,pan2001structured,xia2012superfast,xi2014superfast}. These are based on finding a \emph{representation} of the inverse that has $\Otil(n)$ size (for example the inverse is constructed by shifting and adding a rank two matrix that can be presented by $4$ vectors). The representation is then applied to the response vector of the linear system, for example, using fast Fourier transform (FFT) techniques \cite{ammar1987generalized}. Note that in such super fast algorithms, the inverse is never written explicitly since it costs $\Omega(n^2)$ to write an $n$-by-$n$ matrix explicitly.

Hankel matrices are a special class of \emph{structured matrices} with many connections to other structured matrices such as Toeplitz, generalized Cauchy, and Vandermonde matrices \cite{chun1989fast,pan2001structured}. They also have many applications in theoretical computer science, including solving sparse linear systems \cite{kailath1979displacement,eberly2006solving,eberly2007faster,peng2021solving,nie2022matrix,casacuberta2021faster} (which itself has applications in improving runtime bounds for convex optimization algorithms \cite{cohen2021solving,van2020deterministic,bubeck2018homotopy,adil2019iterative,ghadiri2021sparse}) and sum-of-squares (SoS) decomposition of single variable polynomials \cite{dym1989hermitian,hasan1996hankel,llovet1987hankel}.

A recent breakthrough of Peng and Vempala \cite{peng2021solving} has shown that a poly-conditioned sparse linear system can be solved faster than matrix multiplication time by using block-Krylov methods. The high-level idea is to form a random block-Hankel matrix from the input matrix and then solve a linear system for this Hankel matrix instead. Although the bit complexity of this Hankel matrix is considerably more than the bit complexity of the input matrix (by a factor of $m<n^{0.25}$), Peng and Vempala showed, with a careful analysis, that the number of bit operations of their algorithm is $o(n^{\omega})$ for any $\omega>2$. Note that the algorithm of \cite{peng2021solving} does not generate an explicit inverse but instead generates a linear operator (an implicit inverse) that can be applied to a vector to solve the linear system.

Since the seminal works of Karmarkar \cite{Karmarkar84} and Vaidya \cite{vaidya1989speeding} on solving linear programs (LPs) using interior point methods (IPMs)
maintaining the inverse of a matrix that goes under low-rank updates has been an important tool in improving the running time of algorithms for optimization problems. This inverse maintenance is done using Sherman-Morrison-Woodbury identity (Fact \ref{fact:Woodbury}) which is equivalent to solving a batch of linear systems, i.e., computing $\ma^{-1} \mb$ for a matrix $\mb$ instead of $\ma^{-1} \vb$, which is solving one linear system.

Although the sparse solver of Peng and Vempala is faster than matrix multiplication for solving one linear system, for a batch of linear systems of size $n$ (i.e., $\mb$ is an $n\times n$ matrix), it is slower than direct methods that compute an explicit inverse that can directly be multiplied by $\mb$ \cite{demmel2007fast,DemmelDH07}. Despite this caveat, the sparse solver has been utilized to improve the running time of $p$-norm regression problems for sparse poly-conditioned matrices beyond matrix multiplication time \cite{ghadiri2021sparse}. This improvement crucially depends on the fact that $p$-norm regression, for fixed $p$, can be solved by an algorithm with $\Otil(n^{1/3})$ iterations \cite{adil2019iterative}. A main idea of \cite{ghadiri2021sparse} is to recompute the linear operator associated with the inverse whenever the rank of the update in Sherman-Morrison-Woodbury identity is large and causes the running time to go above $n^{\omega}$. Since the number of iterations is $\Otil(n^{1/3})$, this recomputation only happens a few times and a total running time of $o(n^{\omega})$ is achieved for $p$-norm regression.

This approach, however, does not work for linear programming problems since the IPMs used for these problems require $\Omega(n^{1/2})$ iterations. We provide more details for this issue in Section \ref{sec:rel-work-and-motiv}. Inspired by this, we propose two conjectures regarding the running time of computing symmetric factorizations of the form $\mb \mb^*$ for Hankel matrices and their inverses, where $\mb^*$ denotes the conjugate transpose of the matrix $\mb$. Due to general displacement structures that we will discuss later, these conjectures have implications for the block-Hankel matrix arising in the block-Krylov approach of \cite{peng2021solving}. In particular, the following are implied by our conjectures.

\begin{enumerate}
    \item The first implication of our conjectures is an algorithm for solving a batch of poly-conditioned linear systems faster than \cite{peng2021solving}. We have computed the running times of solving a batch of linear systems for a matrix with polynomial condition number and $O(n)$ nonzero entries using the online tool of Brand \cite{Complexity} that uses the running times developed in \cite{gall2018improved}. This is illustrated in Table \ref{tab:run-times}. For example, for a batch of size $n^{0.96}$ (i.e., computing $\ma^{-1}\mb$ where $\ma\in\R^{n\times n}$ and $\mb\in\R^{n\times n^{0.96}}$), our approach would give an improvement of $n^{0.016}$ in the running time.

    \item Perhaps the most important implication of our conjectures is an algorithm that solves a linear program  with a sufficiently sparse matrix $\ma$ with polynomial condition number faster than matrix multiplication time. The sufficient sparsity is $o(n^{\omega-1})$ nonzero entries. We discuss this in detail in Section \ref{sec:rel-work-and-motiv}.

    \item In addition, the algorithm developed based on our conjectures improves the running time of the sparse $p$-norm regression algorithm developed in \cite{ghadiri2021sparse}.
    
    \begin{table}[t]
    \caption{Comparison of running times of \cite{peng2021solving} with an algorithm implied by our conjecture for solving a batch of linear systems for a poly-conditioned matrix with $O(n)$ nonzero entries.}
        \centering
        \begin{tabular}{|c|c|c|c|c|c|}
        \hline
            Size of batch & $n^{0.95}$ & $n^{0.96}$ & $n^{0.97}$ & $n^{0.98}$ & $n^{0.99}$ \\
            \hline
            Running time by \cite{peng2021solving} & $n^{2.354}$ & $n^{2.357}$ & $n^{2.361}$ & $n^{2.365}$ & $n^{2.369}$ \\
            \hline
            Running time implied by our conjectures & $n^{2.341}$ & $n^{2.341}$ & $n^{2.349}$ & $n^{2.357}$ & $n^{2.365}$ \\
            \hline
        \end{tabular}
        \label{tab:run-times}
    \end{table}
\end{enumerate}

\paragraph{Outline.} Motivated by these applications, we present necessary definitions and preliminaries for understanding our conjectures and results in Section \ref{sec:prelim}. We then present our conjectures and corresponding results that justify them in Section \ref{sec:res-and-conj}. We discuss the applications of Hankel matrices and the implications of our conjectures, including the implications for solving sparse linear programs, in Section \ref{sec:rel-work-and-motiv}. We then provide a result regarding symmetric factorizations of Toeplitz matrices (which is used as a subprocedure for symmetric factorization of Hankel matrices) in Section \ref{sec:sym-toeplitz}. We present a key identity for Hankel matrices in Section \ref{sec:key-identity} that allows us to design a recursive algorithm for symmetric factorization of them. We then present our results regarding the symmetric factorization of Hankel matrices and their inverses in Section \ref{sec:recursive-algo} and \ref{sec:adv-recursive-algo}, respectively. We finally conclude in Section \ref{sec:conclusion}.

\subsection{Notation and Preliminaries}
\label{sec:prelim}

We consider the entries of our matrices to be in a field $\F$. This can be considered the field of reals $\R$ or complex numbers $\C$. For both of these, our factor matrices $\ma$ and $\mb$ are in $\C$. Our results also extend to finite fields $\F$. In this case, the entries of $\ma$ and $\mb$ are from an extension of field $\F$ that contains the square root of all of the elements of $\F$. For $\R$ and $\C$, we consider fixed-point arithmetic for computation and representing our numbers. In this case, we cannot necessarily represent the square root of our numbers with finitely many bits, but for a number $a$ with $\ell$ bits, we can find a number $b$ with $O(\ell)$ bits such that $\abs{b-\sqrt{a}}<2^{-\ell}$. Therefore for matrices over $\R$ and $\C$, our symmetric factorizations have some small error. For matrices over $\R$ and $\C$, we denote the Frobenius norm and the operator norm by $\norm{\cdot}_{\fro}$ and $\norm{\cdot}_{2}$, respectively. Then we define the condition number of an invertible matrix $\ma$ over $\R$ or $\C$ as $\norm{\ma}_2 \cdot \norm{\ma^{-1}}_2$, and we denote it by $\kappa(\ma)$.

We denote the entry $(j,k)$ of a matrix $\mm$ either by $\mm_{j,k}$ or $\mm(j,k)$. For natural numbers $j_2>j_1$ and $k_2>k_1$, we show the block of $\mm$ with rows $j_1,j_1+1,\ldots,j_2$ and columns $k_1,k_1+1,\ldots,k_2$ with $\mm_{j_1:j_2,k_1:k_2}$. The matrix consisting of rows $j_1,\ldots,j_2$ and all columns is denoted by $\mm_{j_1:j_2,:}$. We denote the $n$-by-$n$ identity matrix with $\mi_{n}$ and if the dimension is clear from the context, we drop the subscript. We denote an $m$-by-$n$ matrix of all zeros with $0_{m\times n}$ and if the dimensions are clear from the context, we drop the subscript. We denote the positive definite (Loewner) ordering by $\preceq$. We denote the running time of multiplying an $n\times m$ matrix with an $m\times k$ matrix with $\matmult(n,m,k)$. Then $n^{\omega}=\matmult(n,n,n)$.

We denote the transposition of a matrix $\mm$ by $\mm^\top$ and its conjugate transposition by $\mm^*$. Note that for real matrices, transposition and conjugate transposition are the same. We also denote the complex conjugate of a number $a\in\C$ by $a^*$. Moreover we define $i=\sqrt{-1}$. For a matrix $\mm$, we denote its real part and imaginary part by $\real(\mm)$ and $\imag(\mm)$, respectively. Note that both $\real(\mm)$ and $\imag(\mm)$ are real matrices and $\mm = \real(\mm) + i \cdot \imag(\mm)$.

We use $\Otil$ notation to omit polylogarithmic factors in $n$ and $\ell$ from the complexity, i.e., for function $f$, $\Otil(f):=O(f\cdot \log^c (n\ell))$ where $c$ is a constant. We denote the set $\{1,\ldots,n\}$ by $[n]$.

We extensively use the shift matrix $\mDelta_n \in \F^{n\times n}$ that is zero everywhere except on the entries under the diagonal for which it is one. For example,
\[
\mDelta_4 = \begin{bmatrix}
0 & 0 & 0 & 0 \\
1 & 0 & 0 & 0 \\
0 & 1 & 0 & 0 \\
0 & 0 & 1 & 0
\end{bmatrix}.
\]
When the dimension of $\mDelta_n$ is clear from the context, we omit the subscript and show the shift matrix by $\mDelta$.
Multiplying a matrix from left by $\mDelta$ ($\mDelta^\top$) shifts the rows of the matrix down (up) by one row and multiplying a matrix from right by $\mDelta$ ($\mDelta^\top$) shifts the columns of the matrix left (right) by one column.
A matrix $\mm$ is symmetric if $\mm=\mm^\top$ and is Hermitian if $\mm = \mm^*$. Also $\mm$ is skew-symmetric if $\mm^\top = -\mm$.
Let $\F$ be a field and $\vh = (h_1,\ldots,h_{2n-1}) \in \F^{2n-1}$ be a vector. Then the corresponding Hankel matrix $\mh$ is defined as $\mh_{ij} = h_{i+j-1}$. For example for $n=4$,
\begin{align}
\label{eq:hankel-example}
    \mh = 
    \begin{bmatrix}
    h_1 & h_2 & h_3 & h_4 \\
    h_2 & h_3 & h_4 & h_5 \\
    h_3 & h_4 & h_5 & h_6 \\
    h_4 & h_5 & h_6 & h_7 \\
    \end{bmatrix}.
\end{align}
For a vector $\mt=(t_1,\ldots,t_n)\in\C^{n}$, where $t_1\in\R$, the corresponding Hermitian Toeplitz matrix $\mt$ is defined as $\mt_{i,j} = t_{j-i+1}$ if $j\geq i$, and $\mt_{i,j} = t_{i-j+1}^*$, otherwise. For example, the Hermitian Toeplitz matrix corresponding to $(t_1,t_2,t_3,t_4)$ is
\[
\mt = \begin{bmatrix}
t_1 & t_2 & t_3 & t_4 \\
t_2^* & t_1 & t_2 & t_3 \\
t_3^* & t_2^* & t_1 & t_2 \\
t_4^* & t_3^* & t_2^* & t_1
\end{bmatrix}.
\]
Note that this can be considered for a general field by extending it using the polynomial root $x^2+1=0$.
It is easy to check that for a Toeplitz matrix $\mt$, $\mt - \mDelta \mt \mDelta^\top$ is of rank two, and for a Hankel matrix $\mh$, $\mDelta \mh - \mh \mDelta^\top$ has rank two. These are called the displacement rank of Toeplitz and Hankel matrices. 
The general definitions are as the following.
\begin{definition}[Displacement rank]
Let $\mm,\matu,\mv\in\F^{n\times n}$. The \emph{Sylvester-type displacement rank} of $\mm$ with respect to $(\matu, \mv)$ is equal to the rank of $ \matu \mm - \mm \mv$. The \emph{Stein-type displacement rank} of $\mm$ with respect to $(\matu, \mv)$ is equal to the rank of $\mm - \matu \mm \mv$.
\end{definition}
For example a Hankel matrix $\mh$ has a Sylvester-type displacement rank of \emph{two} with respect to $(\mDelta,\mDelta^\top )$. This allows us to define the displacement rank for block-Hankel matrices of the following form as well.
\[
\mh = \begin{bmatrix}
    \mh_1 & \mh_2 & \mh_3 & \mh_4 \\
    \mh_2 & \mh_3 & \mh_4 & \mh_5 \\
    \mh_3 & \mh_4 & \mh_5 & \mh_6 \\
    \mh_4 & \mh_5 & \mh_6 & \mh_7 \\
    \end{bmatrix},
\]
where each $\mh_i$ is an $s\times s$ matrix. Then $\mh$ has a Sylvester-type displacement rank of $2s$ with respect to $(\matu,\matu^\top)$, where
\[
\matu = \begin{bmatrix}
    0_{s\times s} & 0_{s\times s} & 0_{s\times s} & 0_{s\times s} \\
    \mi_s & 0_{s\times s} & 0_{s\times s} & 0_{s\times s} \\
    0_{s\times s} & \mi_s & 0_{s\times s} & 0_{s\times s} \\
    0_{s\times s} & 0_{s\times s} & \mi_s & 0_{s\times s} \\
    \end{bmatrix}.
\]
Moreover, the inverse of a Hankel matrix is not Hankel but it has a Sylvester-type displacement rank of \emph{two} with respect to $(\mDelta,\mDelta^\top )$. Similarly the inverse of a Toeplitz matrix is not Toeplitz but it has a Stein-type displacement rank of \emph{two} with respect to $(\mDelta,\mDelta^\top)$.

Note that multiplying a vector by a Hankel or Toeplitz matrix can be done in $\Otil(n)$ time using FFT techniques \cite[Chapter 30]{cormen2022introduction} due to their connections to single-variable polynomials. Finally, the following illustrates the connection between inverse maintenance and solving a batch of linear systems.

\begin{fact}[Sherman-Morrison-Woodbury identity \cite{woodbury1950inverting}]\label{fact:Woodbury}
For an invertible $n \times n$ matrix $\mm$ and matrices $\matu\in \R^{n \times r},\md \in \R^{r \times r},\mv \in \R^{r \times n}$, if $\md$ and $(\mm+\matu \md \mv)^{-1}$ are invertible, then
\[
(\mm+\matu \md \mv)^{-1} = \mm^{-1} - \mm^{-1} \matu (\md^{-1} + \mv \mm^{-1} \matu)^{-1} \mv \mm^{-1}.
\]
\end{fact}

\section{Results and Conjectures}
\label{sec:res-and-conj}

Our first conjecture is about computing a symmetric factorization of positive definite Hankel matrix $\mh$ as $\mb \mb^*$ in linear time. Since $\mb$ is at least $n\times n$ (for a full-rank Hankel matrix), we do not require outputting $\mb$ explicitly. Instead, the output should be an implicit representation of size $\Otil(n\cdot \ell)$ that describes $\mb$. Note that this is similar to the way that Hankel matrices are described as well. For example, if we give $(h_1,\ldots,h_7)$ in \eqref{eq:hankel-example}, then the corresponding Hankel matrix is completely described and this representation has a linear size in $n$.

\begin{conjecture}
\label{conj:sym-hankel}
Let $\mh \in \R^{n\times n}$ be a positive definite Hankel matrix with bit complexity $\ell$. There exists an algorithm that finds a representation of a matrix $\mb$ with $n$ rows, $\Otil(n)$ columns, and bit complexity $\ell$ in time $\Otil(n \cdot \ell)$ such that $\norm{\mh - \mb \mb^*}_{\fro}<\frac{1}{2^{\ell}}$. 
\end{conjecture}

Our conjecture over finite fields would require $\mb$ such that $\mh=\mb \mb^*$. In this case, we assume the field operations are performed in $O(1)$ time, and therefore we require a running time of $\Otil(n)$. For matrices over $\C$, we also require $\norm{\mh - \mb \mb^*}_{\fro}<\frac{1}{2^{\ell}}$. To justify our conjecture, we provide an algorithm that runs in the specified running time and computes a representation of a factorization of the form $\mb \mb^* - \mc \mc^*$.

\begin{restatable}{theorem}{symmetricHankel}
\label{thm:sym-hankel}
Let $\mh\in\R^{n\times n}$ be a Hankel matrix with bit complexity $\ell$. There exists an algorithm that finds a representation of matrices $\mb$ and $\mc$, each with $n$ rows, $O(n\log n)$ columns, and bit complexity $\ell$ in time $\Otil(n \cdot \ell)$ such that $\norm{\mh - (\mb \mb^* - \mc \mc^*)}_{\fro}<\frac{1}{2^{\ell}}$.
\end{restatable}

Since Hankel matrices are symmetric, Theorem \ref{thm:sym-hankel} does not require the positive definite condition. Our algorithm gives similar bounds and running times for matrices over $\C$ and over finite fields it finds representations of $\mb$ and $\mc$ such that $\mh = \mb \mb^* - \mc \mc^*$.

The factorization of the form $\mb \mb^* - \mc \mc^*$ has been considered before for Toeplitz matrices and their inverses with the goal of solving linear systems with a Toeplitz matrix in linear time \cite{kailath1978inverses,ammar1987generalized,ammar1988superfast}.
The positive semi-definiteness of $\mb \mb^*$ and $\mc \mc^*$ provides some stability properties for solving linear systems with a Toeplitz matrix \cite{bojanczyk1995stability}.
These algorithms are related to the study of orthogonal polynomials and generally either use the Schur algorithm or Levinson algorithm to compute $\mb$ and $\mc$.
We provide similar results for Toeplitz matrices with a simpler and more straightforward algorithm.

\begin{restatable}{theorem}{symmetricToeplitz}
\label{thm:sym-toeplitz}
Let $\mt\in\R^{n\times n}$ be a Hermitian Toeplitz matrix with bit complexity $\ell$. There exists an algorithm that finds a representation of matrices $\mb$ and $\mc$, each with $n$ rows, $O(n\log n)$ columns, and bit complexity $\ell$  in time $\Otil(n \cdot \ell)$ such that $\norm{\mt - (\mb \mb^* - \mc \mc^*)}_{\fro}<\frac{1}{2^{\ell}}$.
\end{restatable}

Theorem \ref{thm:sym-toeplitz} is used as a subprocedure for Theorem \ref{thm:sym-hankel}.
The simplicity of our algorithm for Toeplitz matrices allows us to use it in combination with a recursive algorithm that recursively decomposes a Hankel matrix to the sum of $\log(n)$ Toeplitz-like matrices to achieve our main result for decomposition of Hankel matrices.

Our second conjecture is about computing a symmetric factorization for the inverse of a positive definite Hankel matrix. Note that we do not require the inverse as input since the approach of \cite{peng2021solving} can obtain a representation of it that has a linear size in linear time.

\begin{conjecture}
\label{conj:sym-inv-hankel}
Let $\mh \in \R^{n\times n}$ be a positive definite Hankel matrix with bit complexity $\ell$ and condition number bounded by $2^\ell$. There exists an algorithm that finds a representation of a matrix $\mb$ with $n$ rows, $\Otil(n)$ columns, and bit complexity $\ell$ in time $\Otil(n^{\omega/2} \cdot \ell)$ such that $\norm{\mh^{-1} - \mb \mb^*}_{\fro}<\frac{1}{2^{\ell}}$. 
\end{conjecture}

Again for finite fields, we require $\mh^{-1}=\mb \mb^*$ and a running time of $\Otil(n^{\omega/2})$. Note the difference between the running time of Conjecture \ref{conj:sym-hankel} and Conjecture \ref{conj:sym-inv-hankel}. This is because of the running time that we can achieve for the factorization of the form $\mb\mb^* - \mc \mc^*$ in the following result.

\begin{restatable}{theorem}{symmetricInverseHankel}
\label{thm:sym-inv-hankel}
Let $\mh \in \R^{n\times n}$ be a Hankel matrix with bit complexity $\ell$ and condition number bounded by $2^\ell$. There exists an algorithm that finds a representation of matrices $\mb$ and $\mc$, each with $n$ rows, $O(n\log n)$ columns, and bit complexity $\ell$ in time $\Otil(n^{\omega/2} \cdot \ell)$ such that $\norm{\mh^{-1} - (\mb \mb^* - \mc \mc^*)}_{\fro}<\frac{1}{2^{\ell}}$.
\end{restatable}

The result of Theorem \ref{thm:sym-inv-hankel} is actually more general than the inverse of Hankel matrices. The algorithm we present can find such a factorization in the specified time for any matrix that has a Sylvester-type displacement rank of two with respect to $(\mDelta, \mDelta^\top)$ and it can be generalized to block matrices as described in Section \ref{sec:prelim}. 

The main reason for the running time difference between Theorem \ref{thm:sym-hankel} and Theorem \ref{thm:sym-inv-hankel} is the recursion in our algorithm. For Theorem \ref{thm:sym-hankel}, our recursion starts with the $n\times n$ Hankel matrix and modifies it to a matrix with four blocks of size $\frac{n}{2}\times \frac{n}{2}$ where each block itself is Hankel, i.e., the displacement rank of the blocks is the same as the larger matrix. It then continues this process for $O(\log n)$ iterations. However, for general matrices with small Sylvester-type displacement rank, when we apply the recursion, the Sylvester-type displacement rank of the blocks is doubled. This forces us to stop the recursion when the size of the blocks is $\sqrt{n}$ and results in the running time proportional to $n^{\omega/2}$.

\section{Motivation and Related Work}
\label{sec:rel-work-and-motiv}

Hankel matrices have many connections to Toeplitz matrices. One can see that reversing the order of rows or columns of a Hankel matrix results in a Toeplitz matrix and vice versa. Therefore solving a linear system for Toeplitz matrices implies a solver for Hankel matrices as well. Therefore many works have focused on Toeplitz matrices. However, there are some applications that are specifically directed to Hankel matrices. Examples are linear system solvers based on block Krylov matrices (that are used to solve linear systems with general poly-conditioned sparse matrices \cite{peng2021solving,nie2022matrix}) and sum-of-squares (SoS) decomposition of single-variable polynomials. 

Here we first discuss sparse linear system solvers based on block-Krylov methods in Section \ref{subsec:faster-batch-solve} and explain how our conjecture leads to faster algorithms for solving a batch of linear systems. Then in Section \ref{subsec:fast-sparse-lp}, we explain how this leads to a faster algorithm for solving sparse poly-conditioned linear programs faster than matrix multiplication time. We finally discuss the connection of Hankel matrices to the sum-of-squares (SoS) decomposition of single variable polynomials in Section \ref{subsec:sos}.

\subsection{Faster Sparse Linear System Solvers for Batch Problems}
\label{subsec:faster-batch-solve}

We start by describing the block-Krylov approach that has resulted in faster sparse linear system solvers for matrices over rational numbers \cite{eberly2006solving,eberly2007faster}, fixed-point arithmetic \cite{peng2021solving,nie2022matrix}, and finite fields \cite{casacuberta2021faster}.

\paragraph{Linear system solvers based on block-Krylov matrices.} 
To solve a linear system $\ma \vx = \vb$, this approach forms a block Krylov matrix 
\[
\mk=\begin{bmatrix}
\mg & \ma \mg & \ma^2 \mg & \cdots & \ma^{m-1} \mg
\end{bmatrix} \in \R^{n\times n},
\] 
where $\mg$ is a sparse $n$-by-$s$ random matrix, and $m\cdot s = n$. If the matrix $\ma$ is sparse, for example, its number of nonzero entries is $O(n)$, then $\mk$ can be formed quickly. Note that $\ma^{i+1} \mg$ can be obtained from $\ma^i \mg$ by multiplying it with $\ma$.
More specifically, for $\ma$ with constant bit complexity, $\mk$ can be formed in time $\Otil(\nnz{\ma} \cdot s \cdot m^2)=\Otil(\nnz{\ma} \cdot n \cdot m)$, where the $\nnz{\ma} \cdot s$ factor comes from the time that takes to multiply $\ma$ by an $n$-by-$s$ matrix. One of the factors of $m$ comes from the number of such matrix multiplications we need to perform and the other one comes from the bit-complexity of the resulting matrices, e.g., the entries of $\ma^{m-1} \mg$ need $\Otil(m)$ bits. For a small enough $m$ (for example, $m\approx n^{0.01}$), $\Otil(\nnz{\ma} \cdot n \cdot m)$ is smaller than the matrix multiplication time if $\nnz{\ma} \ll n^{\omega-1}$.

Then the inverse of $\ma$ is presented by $\mk(\mk^\top \ma \mk)^{-1} \mk^\top$. Note that for symmetric $\ma$, $\mk^\top \ma \mk$ is a block-Hankel matrix of the following form
\[
\mk^\top \ma \mk = \begin{bmatrix}
\mg^\top \ma \mg & \mg^\top \ma^2 \mg & \mg^\top \ma^3 \mg & \cdots & \mg^\top \ma^m \mg \\
\mg^\top \ma^2 \mg & \mg^\top \ma^3 \mg & \mg^\top \ma^4 \mg & \cdots & \mg^\top \ma^{m+1} \mg \\
\mg^\top \ma^3 \mg & \mg^\top \ma^4 \mg & \mg^\top \ma^5 \mg & \cdots & \mg^\top \ma^{m+2} \mg \\
\vdots & \vdots & \vdots & \ddots & \vdots \\
\mg^\top \ma^m \mg & \mg^\top \ma^{m+1} \mg & \mg^\top \ma^{m+2} \mg & \cdots & \mg^\top \ma^{2m-1} \mg \\
\end{bmatrix} \in \R^{n\times n}.
\]
Note that the symmetry assumption for $\ma$ is not a limitation since we can instead consider the linear system $\ma^\top \ma \vx = \ma^\top \vb$, which has a symmetric matrix.
One can think of this matrix as an $m$-by-$m$ Hankel matrix where each entry is a $s$-by-$s$ matrix with bit-complexity of $\Otil(m)$. Therefore multiplying any two entries of this matrix together costs $\Otil(s^\omega \cdot m)$. Moreover $\mk^\top \ma \mk$ can be multiplied with an $n\times s$ matrix in time $\Otil(s^\omega \cdot m^2)$ by using fast Fourier transform (see \cite{peng2021solving} for details). Finally, note that $\mk^\top \ma \mk$ can be formed in time $\Otil(\nnz{\ma} \cdot n \cdot m)$ similar to the approach we described above for computing $\mk$.

\paragraph{Fast and super fast solvers for block-Hankel matrices.}
To discuss the running time of inverting the block-Hankel matrix $\mk^\top \ma \mk$ or applying the inverse to a block-matrix (or a vector), we need to consider the number of block operations. One can think of each block operation as multiplying two blocks of $\mk^\top \ma \mk$ together. These blocks are $s\times s$ and have bit complexity $m$. So multiplying them by fast matrix multiplication \cite{alman2021refined} and using FFT to multiply the corresponding numbers in linear time results in a running time of $\Otil(s^\omega \cdot m)$. Therefore an algorithm that takes $k$ block operations runs in time $\Otil(s^\omega \cdot m \cdot k)$ with the assumption that the bit complexity stays the same during the algorithm.

Therefore fast solvers that need $m^2$ operations are slow for inverting $\mk^\top \ma \mk$, since they result in a total cost of $s^\omega \cdot m^3 > n^{\omega}$. Thus one needs to use super fast solvers for the matrix $\mk^\top \ma \mk$. Most of the classical super fast solvers are either based on orthogonal polynomials \cite{ammar1987generalized,ammar1988superfast} or based on the conversion of Hankel matrix to generalized Cauchy and hierarchically semi-separable (HSS) matrices \cite{xia2012superfast,xi2014superfast} that admit low-rank properties for off-diagonal blocks. The caveat of these methods is that they blow up the bit complexity of $L$ to at least $L^2$. This means an extra factor of $m$ in addition to $s^\omega \cdot m^2$ operations which again results in a total running time of more than $n^{\omega}$. 

There has been another class of super fast solvers based on hierarchical Cholesky decomposition and Schur complements that classically were analyzed in the exact computation setting (for example, for matrices on finite fields) \cite{kailath1979displacement,kailath1995displacement}.
Very recently, \cite{peng2021solving} analyzed such algorithms for real matrices in the fixed-point arithmetic and showed that such super-fast solvers only need to increase the bit complexity by polylogarithmic factors in $n$. This resulted in an algorithm with a total running time of $\Otil(s^\omega m^2)$ for finding a representation of the inverse of $\mk^\top \ma \mk$. This algorithm was one of the main building blocks that allowed \cite{peng2021solving} to go below matrix multiplication time. The representation of the inverse of $\mk^\top \ma \mk$ obtained from this approach is the product of two matrices $\mx$ and $\my^\top$, i.e., 
\begin{align}
\label{eq:pv-representation}
    (\mk^\top \ma \mk)^{-1} \approx \mx \my^\top.
\end{align} 
$\mx$ and $\my$ are block matrices with a small displacement rank of $2s$. Therefore they can be applied to another matrix of size $n\times s$ with $\Otil(m)$ block operations by utilizing FFT.

The caveat of this approach is that the bit complexity of matrices $\mx$ and $\my$ is $\Omega(m)$. Therefore although they can solve one linear system faster than matrix multiplication time, for any selection of parameters $s$ and $m$, there is a $0<c<1$ such that solving $n^{c}$ linear system with a common matrix $\ma$ takes more than the matrix multiplication time. This is strange since inverting the matrix $n$ using fast matrix multiplication takes $\Otil(n^\omega)$ time and then the inverse can be applied to $n$ vectors in $\Otil(n^\omega)$ time \cite{demmel2007fast} and this does not need any sparsity properties. We now bound the running time of solving a batch of linear systems of size $r$ with \cite{peng2021solving} solver. To do so, we need the following lemma (proved in Section \ref{sec:app-proofs}) for the running time of applying the matrix $\mk$ to a matrix of size $n\times r$.

\begin{restatable}{lemma}{applyK}
\label{lemma:applyK}
Let $\ma\in \R^{n\times n}$ and $\mb \in \R^{n\times r}$. Let $\mg \in \R^{n\times s}$ be a matrix with $\Otil(n)$ nonzero entries, where $1\leq s\leq n$ is a divisor of $n$. Let $m=n/s$ and
\[
\mk = \begin{bmatrix}
    \mg & \ma \mg & \ma^2 \mg & \cdots & \ma^{m-1} \mg
\end{bmatrix} \in \R^{n\times n}.
\]
Let bit complexity of $\ma$ and $\mg$ be $\ell$ and the bit complexity of $\mb$ be $m\cdot \ell$. Then $\mk \mb$ and $\mk^\top \mb$ can be computed in $\Otil(\nnz{\ma} \cdot r \cdot m^2 \cdot \ell)$ time.
\end{restatable}

Assuming the bit complexity of the input matrix is constant and its condition number is $\poly{n}$, for a fixed $m$ and $r$, the total running time of applying the inverse operator of \cite{peng2021solving} to an $n\times r$ matrix is the following.

\begin{align}
\label{eq:pv-running-time-1}
\Otil(\nnz{\ma}\cdot n \cdot m + n^{\omega}m^{2-\omega} + m^2 \cdot \matmult(\frac{n}{m},\frac{n}{m},r) + \nnz{\ma}\cdot r \cdot m^2).
\end{align}
The first term of \eqref{eq:pv-running-time-1} is for forming $\mk$ and $\mk^\top \ma \mk$. The second term is finding the representation of the inverse. The third term is the running time of applying the inverse of $\mk^\top \ma \mk$ to an $n\times r$ matrix. The last term is for applying $\mk^\top$ or $\mk$ to an $n\times r$ matrix (see Lemma \ref{lemma:applyK}). Note that for solving one linear system, \eqref{eq:pv-running-time-1} boils down to 
\[
\Otil(\nnz{\ma}\cdot n \cdot m + n^{\omega}m^{2-\omega}).
\]
Then one can see that by taking $m = n \cdot (\nnz{\ma})^{-1/(\omega-1)}$, a running time of 
\[
\Otil(n^2 (\nnz{\ma})^{(\omega-2)/(\omega-1)})
\]
is achieved, which is faster than matrix multiplication for all values of $\omega>2$ and $\nnz{\ma} < n^{\omega-1}$.

The running time of \eqref{eq:pv-running-time-1} is obtained by applying $\mk$, $(\mk^\top \ma \mk)^{-1}$ and $\mk^\top$ separately. Another approach is to take $\mx$ and $\my$ from \eqref{eq:pv-representation} and compute $\mxtil = \mk \mx$ and $\mytil = \mk \my$ using Lemma \ref{lemma:applyK}. Then the inverse of $\ma$ is given by $\mxtil \mytil^\top$, where the bit complexity of $\mxtil$ and $\mytil$ is $\Otil(m)$. Then solving a batch of linear systems of size $r$ by multiplying $\mxtil$ and $\mytil$ takes the following running time.
\begin{align}
\label{eq:pv-running-time-2}
\Otil(\nnz{\ma}\cdot n \cdot m^2 + n^{\omega}m^{2-\omega} + m \cdot \matmult(n,n,r)).
\end{align}
The first term of \eqref{eq:pv-running-time-2} is from computing $\mxtil,\mytil$, which also dominates the running time of forming $\mk$ and $\mk^\top \ma \mk$. The second term is for finding the representation of the inverse of $\mk^\top \ma \mk$, and the last term comes from the running time of multiplying $\mxtil$ and $\mytil$ with an $n\times r$ matrix. Given a fixed $r$, one can optimize over the best value of $m$ for each of \eqref{eq:pv-running-time-1} and \eqref{eq:pv-running-time-2} and report the smaller running time. This is what we used for Table \ref{tab:run-times}.

\paragraph{Symmetric factorization of inverse operator for faster batch solves.}
The main caveat of the approach of \cite{peng2021solving} is that the representation of the inverse has a bit complexity of $\Omega(m)$, whether we use $\mxtil \mytil^{-1}$ representation or $\mk \mx \my^\top \mk^\top$. This is the main reason that when applied to large batches, the running time of \cite{peng2021solving} becomes slower than direct methods. Here we present an approach based on our conjectures to obtain a representation of the inverse with small bit complexity.

Our approach is to write the inverse of $\mk^\top \ma \mk$ as a symmetric factorization $\mx \mx^*$. In this case, the inverse of $\ma$ is represented as $(\mk \mx)(\mk \mx)^*$. Therefore we have
\begin{align}
\label{eq:sym-argument}
\norm{\mk \mx}_{\fro} = \sqrt{\text{trace}((\mk \mx)(\mk \mx)^*)} = \sqrt{\text{trace}(\ma^{-1})} = \sqrt{\sum_{i=1}^n \lambda_i},
\end{align}
where $\lambda_i$'s are the eigenvalues of $\ma^{-1}$. Therefore in the case where $\lambda_i$'s are $\text{poly}(n)$ (which is the assumption in \cite{peng2021solving,nie2022matrix}), the absolute value of entries in the matrix $\mk \mx$ is bounded by $\text{poly}(n)$. Moreover, one can compute $\mk \mx$ using Lemma \ref{lemma:applyK}.
Therefore, in this case, we can represent the inverse of $\ma$ as $\mxtil \mxtil^*$, where $\mxtil = \mk \mx$, and the bit complexity of entries of $\mxtil$ is $\Otil(1)$. Then the running time of solving a batch of linear systems of size $r$ becomes
\begin{align}
\label{eq:our-running-time}
\Otil(\nnz{\ma}\cdot n \cdot m^2 + n^{\omega}m^{1-\omega/2} + \matmult(n,n,r))
\end{align}
since $\mxtil$ can be applied to an $n\times r$ matrix in time $\matmult(n,n,r)$. Note that we require the error bound of less than $1/2^{\ell}$ (which here be less than $1/2^{m}$) in Conjectures \ref{conj:sym-hankel} and \ref{conj:sym-inv-hankel} because the bit complexity of $\mk$ is $\Otil(m)$ and this way we can guarantee that $(\mk \mx)(\mk \mx)^*$ is close to $\ma^{-1}$. Note that Conjecture \ref{conj:sym-inv-hankel} gives an algorithm that runs with $\Otil(m^{\omega/2})$ block operations and uses numbers with the bit complexity of the input problem. Therefore Conjecture \ref{conj:sym-inv-hankel}, if true, computes a representation of the matrix $\mx$ in time $\Otil(s^{\omega} \cdot m^{1+\omega/2}) = \Otil(n^{\omega}m^{1-\omega/2})$ such that 
\[
\norm{\mx \mx^* - \mk^\top \ma \mk}_{\fro} \leq \frac{1}{2^m}.
\]
Since the bit complexity of this representation is $\Otil(m)$, we can write down $\mx$ in time $\Otil(n^2 \cdot m)$ and then use Lemma \ref{lemma:applyK} to compute $\mxtil = \mk \mx$ in time $\Otil(\nnz{\ma} \cdot n \cdot m^2)$.
This gives us the running time stated in 
Equation \eqref{eq:our-running-time}, which is also the formula we used for our running time in Table \ref{tab:run-times}. We next discuss how our approach results in a faster-than-matrix-multiplication time for solving linear programs with sparse and poly-conditioned matrices.

\subsection{Solving Linear Programs Faster than Matrix Multiplication}
\label{subsec:fast-sparse-lp}

Here we first give a simple explanation of the linear systems that are solved in each iteration of interior point methods (IPMs) for solving LPs. IPMs are the state-of-the-art approach for solving LPs. The seminal works of Karmarkar \cite{Karmarkar84} and Vaidya \cite{vaidya1989speeding} started the study of IPMs, and recently, IPM-based approaches have resulted in algorithms that solve linear programs approximately in $\Otil(n^{\omega})$ arithmetic operations \cite{cohen2021solving,van2020deterministic}. We consider the linear programs of the form
\[
\min_{\ma^\top \vx = \vb, \vx\geq 0} \vc^\top \vx ~~~~~ \text{(primal)} ~~~~~   \text{and}  ~~~~~ \max_{\ma \vy\leq \vc} ~~ \vb^\top \vy ~~~~~ \text{(dual)},
\]
where $\ma\in\R^{n\times d}$, $\vb\in\R^{d}$, $\vc \in\R^n$, and $n\geq d$. Starting from a feasible solution, each iteration $k$ of IPM corresponds to computing a vector of the following form
\begin{align}
\label{eq:original-ipm-step}
    \sqrt{\mw^{(k)}} \ma (\ma^\top \mw^{(k)} \ma)^{-1} \ma^\top \sqrt{\mw^{(k)}} \vg^{(k)},
\end{align}
where $\mw^{(k)}\in\R^{n\times n}$ is a diagonal matrix and $\vg\in\R^{n}$. Note that this is equivalent to solving a linear system with the matrix $\ma^\top \mw^{(k)} \ma$. Recent advances in IPMs \cite{cohen2021solving,van2020deterministic,lee2021tutorial} have shown that instead of \eqref{eq:original-ipm-step}, we can use the following vector
\begin{align}
\label{eq:modified-ipm-step}
    \sqrt{\mwtil^{(k)}} \ma (\ma^\top \mwtil^{(k)} \ma)^{-1} \ma^\top \sqrt{\mwtil^{(k)}} \vg^{(k)},
\end{align}
where $\mwtil^{(k)}\in\R^{n\times n}$ is another diagonal matrix such that $\norm{\vw^{(k)} - \vwtil^{(k)}}_{\infty} < C$ for some constant $C$, and $\vw^{(k)}$ and $\vwtil^{(k)}$ are the vectors corresponding to diagonal matrices $\mw^{(k)}$ and $\mwtil^{(k)}$, respectively. Another insight from IPMs is that $\mw^{(k-1)}$ and $\mw^{(k)}$ are very close to each other in the sense that $\norm{\vw^{(k-1)} - \vw^{(k)}}_2 < \beta$ for some constant $\beta$. Then the following lemma allows us to bound the number of low-rank changes we need to apply to $\mwtil^{(k)}$ to maintain $\norm{\vw^{(k)} - \vwtil^{(k)}}_{\infty} < C$ over the course of the algorithm. Therefore we can use the Sherman-Morrison-Woodbury identity (Fact \ref{fact:Woodbury}) to maintain the inverse $(\ma^\top \mwtil^{(k)} \ma)^{-1}$, and this results in an algorithm for solving LPs with $\Otil(n^{\omega})$ arithmetic operations.
 
\begin{lemma}[\cite{lee2021tutorial}]
\label{lemma:robust-ipm-num-change}
Let $\beta>0$ be a constant.
Let $\vecv^{(0)},\vecv^{(1)},\vecv^{(2)},\ldots$ be vectors in $\R^{n}$ arriving in a stream with the guarantee that $\norm{\vecv^{(k+1)}-\vecv^{(k)}}_2 \leq \beta$ for all $k$. Then for $0<C<0.5$, we can pick $\vecvbar^{(0)},\vecvbar^{(1)},\vecvbar^{(2)},\ldots$, so that (see Algorithm 4 on \cite{lee2021tutorial})
\begin{itemize}
    \item $\norm{\vecvbar^{(k)}-\vecv^{(k)}}_{\infty} \leq C$ for all $k$.
    \item $\norm{\vecvbar^{(k)}-\vecvbar^{(k-1)}}_0 \leq O(2^{2q_k}(\beta/C)^2 \log^2(n))$ where $q_k$ is the largest integer with $k = 0 \mod 2^{q_k}$.
\end{itemize}
\end{lemma}

In the original papers of Cohen-Lee-Song \cite{cohen2021solving} and Brand \cite{van2020deterministic}, instead of the matrix $(\ma^\top \mwtil^{(k)} \ma)^{-1}$, the matrix $\ma(\ma^\top \mwtil^{(k)} \ma)^{-1} \ma^\top$ is maintained. The reason is that for a dense matrix (e.g., $\nnz{\ma}=\Omega(n^2)$), the cost of multiplying $\ma$ by a vector in each iteration is $\Omega(n^2)$. Therefore since the number of iterations of IPM is $\sqrt{n}$, this alone gives a running time of $\Omega(n^{2.5})$, which is much higher than $n^{\omega}$. However in our case, since $\ma$ is sparse with $o(n^{\omega-1})$ nonzero entries, the cost of this multiplication over the course of the algorithm is at most $O(n^{\omega-0.5})$. Therefore we focus on maintaining $(\ma^\top \mwtil^{(k)} \ma)^{-1}$. 

To maintain $(\ma^\top \mwtil^{(k)} \ma)^{-1}$, we either have to use Fact \ref{fact:Woodbury} or compute $(\ma^\top \mwtil^{(k)} \ma)^{-1}$ from scratch. Consider a fix $m$ for the sparse solver of \cite{peng2021solving} and the number of updates of rank $\frac{n}{m}$ in the IPM, i.e., the number of indices $k$ such that $\frac{n}{m}$ entries are different between $\mwtil^{k}$ and $\mwtil^{k+1}$. By Lemma \ref{lemma:robust-ipm-num-change}, the number of such changes is $\Otil(\sqrt{m})$. If we recompute the inverse from scratch when we encounter these updates, then by \eqref{eq:pv-running-time-1}, our cost is at lease
\[
\Omega(\nnz{\ma}\cdot n \cdot m + n^{\omega}m^{2.5-\omega}),
\]
which is larger than $n^{\omega}$ because $2.5>\omega$. If we use Sherman-Morrison Woodbury identity (Fact \ref{fact:Woodbury}), since it is equivalent to applying the inverse to an $n\times \frac{n}{m}$ matrix, the cost is at least
\[
\Omega(m^{0.5} \cdot s^{\omega} \cdot m^2) = \Omega(s^{\omega} \cdot m^{2.5}),
\]
because applying the inverse of $\mk^\top \ma \mk$ to an $n\times \frac{n}{m}$ matrix costs at least $\Omega(s^{\omega} \cdot m^2)$. This is again more than $n^{\omega}$ because $ms=n$. Using the $\mxtil \mytil $ representation also leads to a cost of 
\[
\Omega(m^{1.5}\cdot \matmult(n,n,\frac{n}{m})),
\]
which is again more than $n^{\omega}$. However, if based on our conjectures, we had a representation of the form $\mxtil \mxtil^*$, then by \eqref{eq:our-running-time}, the cost of this would be
\[
O(m^{0.5}\cdot \matmult(n,n,\frac{n}{m})),
\]
which is smaller than matrix multiplication time. Now suppose our conjectures are true and we can find a representation of the inverse as $\mxtil \mxtil^*$. To go below matrix multiplication time for this inverse maintenance problem, one can adapt the following approach: If the rank of the update is larger than $\frac{n}{m^{(\omega-2)/2}}$, recompute the inverse and $\mxtil \mxtil^*$ from scratch. If the rank of the update is smaller than $n^{\alpha}$, use the Sherman-Morrison-Woodbury identity in an online way, i.e., compute the product of each term with the given vector separately (where $\alpha>0.31$ is the dual of matrix multiplication exponent and is the largest number such that an $n\times n$ matrix can be multiplied with an $n\times n^{\alpha}$ matrix in $O(n^{2+o(1)})$ time). Finally if the rank of the update was between $n^{\alpha}$ and $\frac{n}{m^{(\omega-2)/2}}$, compute the update term of Sherman-Morrison-Woodbury identity (i.e., the second term) and store it as an explicit matrix $\mq$.

With the above approach, the inverse operator is then given as $\mxtil \mxtil^* + \mq + \mm$, where $\mm$ is an implicit matrix given by Sherman-Morrison-Woodbury identity, i.e., 
\[
\mm = - (\mxtil \mxtil^* + \mq) (\ma_S)^\top (\md^{-1} + \ma_S (\mxtil \mxtil^* + \mq) (\ma_S)^\top)^{-1} \ma_S (\mxtil \mxtil^* + \mq),
\]
where $S$ is the set of indices corresponding to updates to $\mwtil$ that are not incorporated to $\mxtil \mxtil^*$ or $\mq$, and $\md$ is the diagonal matrix corresponding to these updates.
Note that the cost of applying $\mq$ to any matrix is the same as the cost of applying $\mxtil$.

Then one can see that by Lemma \ref{lemma:robust-ipm-num-change} and Equation \eqref{eq:our-running-time}, the cost of inverse maintenance is bounded by 
\begin{align}
\label{eq:our-inv-maintenance-run-time}
\Otil(\nnz{\ma}\cdot n \cdot m^{1.5+\omega/4} + n^{\omega}m^{0.5-\omega/4} + m^{(\omega-2)/4}\matmult(n,n,\frac{n}{m^{(\omega-2)/2}})),
\end{align}
where the first two terms come from the cost of recomputations of the inverse, and the other term comes from the updates performed using Sherman-Morrison-Woodbury identity. Now note that the exponent of $m$ in the second term of \eqref{eq:our-inv-maintenance-run-time} is negative for any $\omega>2$. For any value of $\nnz{\ma}\ll n^{\omega-1}$, we can take $m$ small enough to make the first term less than $n^{\omega}$. Similarly, the third term is smaller than $n^{\omega}$. This can be checked by the online tool of \cite{Complexity}.

In addition to inverse maintenance, one needs to consider the cost of queries for computing \eqref{eq:modified-ipm-step}. For these, since we only need to have $\mwtil^{(k)}$ that is close to $\mw^{(k)}$, one does not need to compute all the entries of the formula. We only need to compute the entries that cause an entry of $\mwtil^{(k)}$ to change in the next iteration. This can be done by using heavy-hitters data structures in a way similar to their use in the recent works for solving tall dense linear programs, see \cite{BrandLSS20,BLLSS0W21}. We omit the details of this here, but one can verify that with this approach, the total cost of queries can also be made less than matrix multiplication time. Therefore the overall approach gives an algorithm for solving linear programs with a sparse (i.e., $\nnz{\ma}=o(n^{\omega-1})$) and poly-conditioned matrix faster than matrix multiplication time.

\subsection{SoS decomposition of polynomials}
\label{subsec:sos}
If the coefficients of a degree $n$ polynomial $p$ is represented by a vector 
\[\va = \begin{bmatrix}
a_{n} & a_{n-1} & \cdots & a_1 & a_0
\end{bmatrix}^\top \in \F^{n+1},
\] then with $\vx = \begin{bmatrix}
x^{n} & x^{n-1} & \cdots & x & 1
\end{bmatrix}^\top$, $\vc^\top \vx = p$.
Another way of representing a polynomial $p(x)=a_0 + a_1 x + a_2 x^2 + \cdots + a_{2k} x^{2k}$ of \emph{even} degree using a Hankel matrix is to define $\mh\in\F^{(k+1)\times (k+1)}$ as $\mh_{ij} = \frac{a_{i+j-2}}{i+j-1}$ if $i+j \leq k+1$, and $\mh_{ij} = \frac{a_{i+j-2}}{2k-i-j+1}$, otherwise. For example, for a degree $4$ polynomial, we have
\[
\mh = \begin{bmatrix}
a_0 & \frac{a_1}{2} & \frac{a_2}{3} \\ 
\frac{a_1}{2} & \frac{a_2}{3} & \frac{a_3}{2} \\ 
\frac{a_2}{3} & \frac{a_3}{2} & a_4
\end{bmatrix}.
\]
Then one can see that with $\vx = \begin{bmatrix}
1 & x & \cdots & x^{k-1} & x^k
\end{bmatrix}^\top$, we have $p=\vx^\top \mh \vx$. Now suppose there exists polynomials $\ell_1,\ldots,\ell_m$ (of degree at most $k$) such that $p = \sum_{j=1}^m \ell_j^2$. Then showing the coefficient of $\ell_j$ with $b_{0}^{(j)},\ldots,b_{k}^{(j)}$, for $j\in[m]$, and defining the matrix $\mb\in\F^{(k+1)\times m}$ as $\mb_{r,j}=b_{r-1}^{(j)}$,
we have $p = \vx^\top \mb \mb^\top \vx$.
Now note that a symmetric factorization of $\mh$ like $\mh=\mb \mb^\top$ gives us such coefficients for the polynomials. Moreover for $j\in[m]$, we have
\begin{align*}
\ell_j^2 (x) & =\vx^\top\mb_{:,j}^\top \mb_{:,j} \vx=\vx^\top \begin{bmatrix}
b_{0}^{(j)} & b_{1}^{(j)} & \cdots & b_{k}^{(j)}
\end{bmatrix}^\top \begin{bmatrix}
b_{0}^{(j)} & b_{1}^{(j)} & \cdots & b_{k}^{(j)}
\end{bmatrix} \vx
\\ & = (b_{0}^{(j)} + b_{1}^{(j)} x + \cdots + b_{k}^{(j)} x^k)^2
\end{align*}
Therefore symmetric factorization of Hankel matrices give a sum-of-squares (SoS) decomposition of single-variable polynomials.

\section{Symmetric Factorization of Hermitian Toeplitz Matrices}
\label{sec:sym-toeplitz}

We start this section by showing how one can find a symmetric factorization of a certain rank-two Hermitian matrix. We will then use this to find a symmetric factorization for a Hermitian Toeplitz matrix.

\begin{lemma}
\label{lemma:rank2-symmetric}
Let $\mm$ be a rank two Hermitian matrix of the following form.
\begin{align}
\label{eq:rank-2-mat}
\mm = \begin{bmatrix}
0 & \cdots & 0 & t_1^* & 0 & \cdots & 0 \\
0 & \cdots & 0 & t_2^* & 0 & \cdots & 0 \\
\vdots & \ddots & \vdots & \vdots & \vdots & \ddots & \vdots \\
0 & \cdots & 0 & t_{j-1}^* & 0 & \cdots & 0 \\
t_1 & \cdots & t_{j-1} & 0 & t_{j+1} & \cdots & t_{n} \\
0 & \cdots & 0 & t_{j+1}^* & 0 & \cdots & 0 \\
\vdots & \ddots & \vdots & \vdots & \vdots & \ddots & \vdots \\
0 & \cdots & 0 & t_{n}^* & 0 & \cdots & 0
\end{bmatrix}.
\end{align}
In other words, only the $j$'th row and column of this matrix is nonzero and its entry $(j,j)$ is also zero. Then $\mm$ has exactly two nonzero eigenvalues $\lambda_1$ and $\lambda_2$ that are real and $\lambda_1=-\lambda_2$. Moreover let $\vvtil_1$ and $\vvtil_2$ be the eigenvectors corresponding to $\lambda_1$ and $\lambda_2$, respectively. Also let $\lambda_1$ be the positive eigenvalue and $\vv_1=\sqrt{\lambda_1} \vvtil_1, \vv_2=\sqrt{\lambda_1} \vvtil_2$. Then $\mm = \vv_1 \vv_1^* - \vv_2 \vv_2^*$.
\end{lemma}
\begin{proof}
We calculate the eigenvalue decomposition of $\mm$. Since $\mm$ is Hermitian, its eigenvectors can be picked to be orthonormal, and since $\mm$ is a rank two matrix, it has at most two nonzero eigenvalues that can be computed by the formula $\mm \vv=\lambda \vv$. This gives the following set of linear systems
\begin{align}
\label{eq:eig-system}
\nonumber
\sum_{k\in[n],k\neq j} t_k v_k & = \lambda v_j, \\
t_{k}^* v_j & = \lambda v_{k}, \forall k\in[n], k\neq j.
\end{align}
Therefore for nonzero $\lambda$, we have $v_k=\frac{t_k^* v_j}{\lambda}$. Substituting this into the first equation, we have 
\[
\frac{v_j}{\lambda}\sum_{k\in[n],k\neq j}^n (t_k t_k^*) = \lambda v_j.
\]
Note that $v_1$ is nonzero because otherwise all of $v_k$'s are zero by \eqref{eq:eig-system} (and this is in contrast with the assumption that the norm of the eigenvectors is equal to one). Therefore 
\[
\lambda^2 = \sum_{k\in[n],k\neq j}^n (t_k t_k^*) = \sum_{k\in[n],k\neq j}^n \norm{t_k}^2.
\] 
Hence the right hand side is positive and $\mm$ has two \emph{real} eigenvalues $\lambda_1=\sqrt{\sum_{k=1}^n (t_k t_k^*)}$ and $\lambda_2=-\sqrt{\sum_{k=1}^n (t_k t_k^*)}$, where we define $t_j=0$. Let $\vvtil_1$ and $\vvtil_2$ be the eigenvectors corresponding to $\lambda_1$ and $\lambda_2$, respectively. Let $\vv_1=\sqrt{\lambda_1} \vvtil_1$ and $\vv_2=\sqrt{\lambda_1} \vvtil_2$. Note that since $\lambda_1$ is positive $\sqrt{\lambda_1}$ is real and therefore $\vv_1^* = \sqrt{\lambda_1} \vvtil_1^*$ and $\vv_2^* = \sqrt{\lambda_1} \vvtil_2^*$. Then we have
\[
\mm = \lambda_1 \vvtil_1 \vvtil_1^* + \lambda_2 \vvtil_2 \vvtil_2^* = \lambda_1 \vvtil_1 \vvtil_1^* - \lambda_1 \vvtil_2 \vvtil_2^* = \vv_1 \vv_1^* - \vv_2 \vv_2^*.
\]
\end{proof}

By Lemma \ref{lemma:rank2-symmetric}, to find a symmetric factorization of the matrix $\mm$ in \eqref{eq:rank-2-mat}, we only need to find its eigenvalues and eigenvector. Since $\mm$ is a rank two matrix with $O(n)$ nonzero entries, this can be done in $O(n)$ time. To prove Theorem \ref{thm:sym-toeplitz}, we essentially find a symmetric factorization of such a matrix and show that a symmetric factorization of a Hermitian Toeplitz matrix can be constructed by shifting and adding this symmetric factorization for a rank two matrix.

\symmetricToeplitz*
\begin{proof}
Let $\mttil$ be a Toeplitz matrix that is equal to $\mt$ everywhere except on the diagonal and the diagonal of $\mttil$ is equal to zero. We now show that $\mttil$ can be written as $\mv_1 \mv_1^* - \mv_2 \mv_2^*$ for $\mv_1,\mv_2\in\F^{n\times n}$.

Let $\mm=\mttil - \mDelta \mttil \mDelta^\top$. Since $\mttil$ is a Toeplitz matrix, $\mm$ is a matrix of rank two of the following form
\[
\mm = 
\begin{bmatrix}
0 & t_2 & t_3 & \cdots & t_{n} \\
t_2^* & 0 & 0 & \cdots & 0 \\
t_3^* & 0 & 0 & \cdots & 0 \\
\vdots & \vdots & \vdots & \ddots & \vdots \\
t_{n}^* & 0 & 0 & \cdots & 0
\end{bmatrix}.
\]
Therefore by Lemma \ref{lemma:rank2-symmetric}, there exists $\vv_1$ and $\vv_2$ such that $\mm = \vv_1 \vv_1^* - \vv_2 \vv_2^*$.
Now note that
\[
\mttil = \sum_{j=1}^{n} \mDelta^{j-1} \mm (\mDelta^{j-1})^\top.
\]
Therefore defining
\[
\mv_1 = \begin{bmatrix}
\vv_1 & \mDelta \vv_1 & \mDelta^2 \vv_1 & \cdots & \mDelta^{n-1} \vv_1
\end{bmatrix},
\]
\[ 
\mv_2 = \begin{bmatrix}
\vv_2 & \mDelta \vv_2 & \mDelta^2 \vv_2 & \cdots & \mDelta^{n-1} \vv_2
\end{bmatrix},
\]
we have $\mttil = \mv_1 \mv_1^* - \mv_2 \mv_2^*$. Let $t_1$ be the diagonal element of $\mt$ that is a real number because $\mt$ is a Hermitian matrix. Then if $t_1\geq 0$, setting
\[
\mb = \begin{bmatrix}
\mv_1 & \sqrt{\abs{t_1}} \mi
\end{bmatrix}, \text{ and } \mc = \mv_2,
\]
and setting 
\[
\mb = \mv_1, \text{ and } \mc = \begin{bmatrix}
\mv_2 & \sqrt{\abs{t_1}} \mi
\end{bmatrix},
\]
otherwise, we have $\mt = \mb \mb^* - \mc \mc^*$. Finally note that in the first case $\mc$ is a Toeplitz matrix, and therefore has a displacement rank of two and $\mb$ has a displacement rank of two with respect to $(\mDelta, \begin{bmatrix}
\mDelta & \mDelta
\end{bmatrix}^\top)$. Similarly, in the second case also the displacement rank of both matrices is two. This implies that a vector can be multiplied by $\mb,\mc,\mb^*,\mc^*$ in $\Otil(n)$ time by FFT techniques.
\end{proof}

\section{Key Identity for Hankel Matrices}
\label{sec:key-identity}

In this section, we consider a symmetric factorization of a Hankel matrix $\mh$ and without loss of generality, we assume $\mh\in\F^{2^k\times 2^k}$, for $k\in\N$. Note that if the dimensions of $\mh$ is not a power of two, we can extend it to a Hankel matrix $\mhtil$ in which the dimensions are a power of two as the following. Let $\vh=(h_1,\ldots,h_{2s-1})$, for $s\in\N$, be the generating vector of the Hankel matrix $\mh$. In this case $\mh\in\F^{s\times s}$, and we are assuming  $s$ is not a power of two. Let $k$ be the smallest integer such that $2^k>s$ and let $\vhtil=(h_1,\ldots,h_{2s-1},0,\ldots,0)\in\F^{2^{k+1}-1}$. Now let $\mhtil$ be a Hankel matrix with generating vector $\vhtil$. Then $\mhtil \in \F^{2^k \times 2^k}$. Moreover $\mhtil_{1:s,1:s} = \mh$. Therefore if $\mbtil,\mctil$ are matrices such that $\mhtil = \mbtil \mbtil^* - \mctil \mctil^*$, then defining $\mb = \mbtil_{1:s,:}$ and $\mc = \mctil_{1:s,:}$, we have 
$\mh = \mb \mb^* - \mc \mc^*$. Therefore we only need to find a symmetric factorization of $\mhtil$.
We now define a matrix that converts a Hankel matrix to Toeplitz and vice versa.

\begin{definition}
\label{def:j}
Let $\mj_{n} \in \F^{n\times n}$ be a matrix with $\mj_{n}(i,j)=1$ if $i+j=n+1$, and $\mj_{n}(i,j)=0$, otherwise. For example
\[
\mj_{4} = \begin{bmatrix}
0 & 0 & 0 & 1 \\
0 & 0 & 1 & 0 \\
0 & 1 & 0 & 0 \\
1 & 0 & 0 & 0
\end{bmatrix}.
\]
We call this matrix the exchange matrix (also called backward identity). When the dimension is clear from the context, we show the exchange matrix with just $\mj$. Note that $\mj \mj = \mi$, and $\mj^\top = \mj$. Moreover we say, a matrix $\mm$ is centrosymmetric if $\mj \mm = \mm \mj$, is persymmetric if $\mm \mj = \mj \mm^\top$, and 
is bisymmetric if it is both symmetric and centrosymmetric.
\end{definition}

The next lemma describes our similarity transformation to decompose a Hankel matrix to the sum of a Hermitian Toeplitz matrix and a centrosymmetric Hankel matrix.

\begin{lemma}[Key Identity]
\label{lemma:key-identity}
Let $\ms = \frac{1}{2}(1+i)\mi+ \frac{1}{2}(1-i) \mj$, where $\mi$ is the identity matrix and $\mj$ is the exchange matrix. Let $\mh$ be a Hankel matrix.
The imaginary part of $\ms \mh \ms^*$ is a skew-symmetric Toeplitz matrix with zero diagonal and its real part is a centrosymmetric Hankel matrix.
\end{lemma}
\begin{proof}
We have
\begin{align*}
\ms \mh \ms^* & = (\frac{1}{2}(1+i)\mi+ \frac{1}{2}(1-i) \mj) \mh (\frac{1}{2}(1+i)\mi+ \frac{1}{2}(1-i) \mj)^*
\\ & =
(\frac{1}{2}(1+i)\mi+ \frac{1}{2}(1-i) \mj) \mh (\frac{1}{2}(1-i)\mi+ \frac{1}{2}(1+i) \mj)
\\ & = 
\frac{1}{4}((1+i)(1-i)\mh + (1+i)(1-i)\mj \mh \mj + (1+i)^2\mh \mj + (1-i)^2 \mj \mh)
\\ & =
\frac{1}{2} (\mh + \mj \mh \mj) + \frac{i}{2} (\mh \mj - \mj \mh).
\end{align*}
Therefore the real part of $\ms \mh \ms^*$ is $\frac{1}{2} (\mh + \mj \mh \mj)$. Now we have
\[
(\mh + \mj \mh \mj) \mj = \mh \mj + \mj \mh = \mj( \mj \mh \mj + \mh).
\]
Therefore $\real(\ms \mh \ms^*)$ is centrosymmetric. Also note that both $\mh$ and $\mj \mh \mj$ are Hankel and the sum of Hankel matrices is a Hankel matrix. Therefore $\real(\ms \mh \ms^*)$ is also Hankel. In addition, note that since $\mh$ is Hankel, both $\mh \mj$ and $\mj \mh$ are Toeplitz matrices and the sum (and also the difference) of Toeplitz matrices, is a Toeplitz matrix. Therefore $\mh \mj - \mj \mh$ and $\imag(\ms \mh \ms^*)$ are Toeplitz matrices. Finally we have
\[
(\mh \mj - \mj \mh)^\top = \mj^\top \mh^\top - \mh^\top \mj^\top = \mj \mh - \mh \mj = -(\mh \mj - \mj \mh).
\]
Therefore $\imag(\ms \mh \ms^*)$ is a skew-symmetric matrix, and hence its diagonal is equal to zero. Note that this implies $i\cdot \imag(\ms \mh \ms^*)$ is a Hermitian matrix.
\end{proof}

\noindent
We now give a $4$-by-$4$ example to understand Lemma \ref{lemma:key-identity} better. Let
\[
\mh = 
\begin{bmatrix}
h_1 & h_2 & h_3 & h_4 \\
h_2 & h_3 & h_4 & h_5 \\
h_3 & h_4 & h_5 & h_6 \\
h_4 & h_5 & h_6 & h_7 \\
\end{bmatrix}.
\]
We then have
\[
\real(\ms \mh \ms^*) =\frac{1}{2}\begin{bmatrix}
h_1+h_7 & h_2+h_6 & h_3+h_5 & 2h_4 \\
h_2+h_6 & h_3+h_5 & 2 h_4 & h_3+ h_5 \\
h_3 + h_5 & 2h_4 & h_3+h_5 & h_2+h_6 \\
2h_4 & h_3+h_5 & h_2+h_6 & h_1+h_7 \\
\end{bmatrix},
\]
\[
\imag(\ms \mh \ms^*) = \frac{1}{2} \begin{bmatrix}
0 & h_3 - h_5 & h_2 - h_6 & h_1 - h_7 \\
h_5 - h_3 & 0 & h_3 - h_5 & h_2 - h_6 \\
h_6 - h_2 & h_5 - h_3 & 0 & h_3 - h_5 \\
h_7 - h_1 & h_6 - h_2 & h_5 - h_3 & 0
\end{bmatrix}.
\]

\noindent
Now note that matrix $\ms$ is a unitary matrix and therefore $\ms \ms^* = \ms^* \ms = \mi$. Therefore, we have
\begin{align}
\label{eq:h-iden}
\mh = \ms^* \ms \mh \ms^* \ms = \ms^* \left( \real(\ms \mh \ms^*) + i\cdot \imag(\ms \mh \ms^*) \right) \ms
\end{align}
\noindent
Therefore if we have matrices $\mb_1,\mb_2,\mc_1,\mc_2$ such that $\real(\ms \mh \ms^*) = \mb_1 \mb_1^* - \mc_1 \mc_1^*$ and $i\cdot\imag(\ms \mh \ms^*) = \mb_2 \mb_2^* - \mc_2 \mc_2^*$, then we have
\[
\mh = \begin{bmatrix}
\ms^* \mb_1 & \ms^* \mb_2
\end{bmatrix}
\begin{bmatrix}
\ms^* \mb_1 & \ms^* \mb_2
\end{bmatrix}^* - \begin{bmatrix}
\ms^* \mc_1 & \ms^* \mc_2
\end{bmatrix}
\begin{bmatrix}
\ms^* \mc_1 & \ms^* \mc_2
\end{bmatrix}^*.
\]
In the next section, we discuss how Lemma \ref{lemma:key-identity} can be exploited to devise our recursive algorithm.

\RestyleAlgo{algoruled}
\IncMargin{0.15cm}
\begin{algorithm}[t]
\textbf{Input:} Hankel matrix $\mh\in \F^{2^k\times 2^k}$ \\
Set $\mh_0 = \mh$ \\
\For {$t=1,\ldots,k$}{
Set $\mn_t = i \cdot \imag(\mstil_{t} \mh_{t-1} \mstil_{t}^*)$ \\
Set $\mh_t = \real(\mstil_{t} \mh_{t-1} \mstil_{t}^*)$ \\
}
Set $\mn_{k+1} = \mh_k$ \\
\For {$t=1,\ldots,k+1$}{
Set $\mx_t$ and $\my_t$ to be matrices such that $\mn_t = \mx_t \mx_t^* - \my_t \my_t^*$ \\
Set $\mb_t = \mstil_1^* \mstil_2^* \cdots \mstil_{t-1}^* \mstil_t^* \mx_t$ and $\mc_t = \mstil_1^* \mstil_2^* \cdots \mstil_{t-1}^* \mstil_t^* \my_t$
}
\Return $\begin{bmatrix}
\mb_{k+1} & \mb_{k} & \cdots & \mb_2 & \mb_1
\end{bmatrix}$ and $\begin{bmatrix}
\mc_{k+1} & \mc_{k} & \cdots & \mc_2 & \mc_1
\end{bmatrix}$
\caption{Symmetric Factorization of Hankel Matrices}
\label{alg:hankel-sym-main-algo}
\end{algorithm}

\section{Symmetric Factorization of Hankel Matrices}
\label{sec:recursive-algo}

Since $i\cdot \imag(\ms \mh \ms^*)$ in \eqref{eq:h-iden} is a Hermitian Toeplitz matrix, we can use Theorem \ref{thm:sym-toeplitz} to find a symmetric factorization for it. 
To deal with the real part of $\ms \mh \ms^*$, we use Lemma \ref{lemma:key-identity} in a recursive fashion using the following matrix. 

\begin{definition}
For $n=2^k$, and $t=1,\ldots,k$, we define $\ms_t := \frac{1}{2}(1+i)\mi_{n/2^{t-1}}+ \frac{1}{2}(1-i) \mj_{n/2^{t-1}} \in \F^{(2^{k-t+1})\times (2^{k-t+1})}$, and
\[
\mstil_t := \begin{bmatrix}
\ms_t & 0 & \cdots & 0 \\
0 & \ms_t & \cdots & 0 \\
\vdots & \vdots & \ddots & \vdots \\
0 & 0 & \cdots & \ms_t
\end{bmatrix} \in \F^{2^k\times 2^k}.
\]
\end{definition}

Algorithm \ref{alg:hankel-sym-main-algo} is our main procedure to find a symmetric factorization of a Hankel matrix. However, so far, we only know how to find a symmetric factorization of $\mn_1$ using Theorem \ref{thm:sym-toeplitz}. Therefore we discuss how to find a symmetric factorization for the rest of $\mn_t$'s. We start by characterizing the structure of matrices $\mh_t$ and $\mn_t$.

\begin{lemma}
\label{lemma:h-struct}
For $t=1,\ldots,k$, matrix $\mh_t$ consists of $2^{t}\times 2^t$ blocks of Hankel matrices as the following. The first block-row consists of $2^t$ Hankel matrices
\[
\begin{bmatrix}
\mq_1 & \mq_2 & \cdots & \mq_{2^{t}-1} & \mq_{2^t}
\end{bmatrix}.
\]
The second block-row is
\[
\begin{bmatrix}
\mq_2 & \mj \mq_1 \mj & \mq_4 & \mj \mq_3 \mj & \cdots & \mq_{2^t} & \mj \mq_{2^t - 1} \mj 
\end{bmatrix}
\]
For $s=1,\ldots,t-1$, let $\mp_{s}$ be the matrix consisting of block-rows $1$ to $2^s$, and $\mptil_{s}$ be the matrix consisting of block-rows $2^{s}+1$ to $2^{s+1}$. Let
\[
\mp_s =
\begin{bmatrix}
\mp_{s,1} & \mp_{s,2} & \cdots & \mp_{s,2^{t-s}}
\end{bmatrix}.
\]
Then
\[
\mptil_{s} = \begin{bmatrix}
\mp_{s,2} & \mp_{s,1} & \mp_{s,4} & \mp_{s,3} & \cdots & \mp_{s,2^{t-s}} & \mp_{s,2^{t-s}-1}
\end{bmatrix}.
\]
Moreover the structure of the block-columns of $\mh_t$ is similar to the structure of the block-rows we described above.
\end{lemma}
\begin{proof}
Before proving the lemma, note that the description completely describes $\mh_t$ (at least up to the block structure) since it describes the first two block-rows and then it uses the first $2^s$ block-rows to describe the next $2^s$ block rows. The structure of $\mh_1$ follows from Lemma \ref{lemma:key-identity}. We then use induction to prove the structure for the rest of $\mh_t$'s.

First note that the number of block-rows and block-columns of $\mh_{t}$ is twice the number of block-rows and block-columns of $\mh_{t-1}$. In other words, each block of $\mh_{t-1}$ is split into four blocks in $\mh_{t}$. Note that in iteration $t$, we multiply $\mh_{t-1}$ by $\mstil_{t}$ and $\mstil_{t}^*$ from left and right, respectively, which is equivalent to multiplying each block of $\mh_{t-1}$ by $\ms_{t}$ and $\ms_{t}^*$ from left and right, respectively. Therefore by induction hypothesis for $\mh_{t-1}$ and Lemma \ref{lemma:key-identity}, since the blocks of $\mh_{t-1}$ are Hankel matrices, the blocks of $\mh_{t}=\real(\mstil_{t} \mh_{t-1} \mstil_{t}^*)$ are also Hankel matrices. Moreover the relation between the first block-row and the second block-row directly follows from Lemma \ref{lemma:key-identity} due to centrosymmetry of the resulting Hankel matrix. Finally the relation between $\mp_s$ and $\mptil_s$ for $s=1,\ldots,t-1$ simply follows from the induction hypothesis for the structure of $\mh_{t-1}$. A similar argument proves the structure of block-columns as well.
\end{proof}

\noindent
We now use Lemma \ref{lemma:h-struct} to characterize the structure of matrices $\mn_t$.

\begin{lemma}
\label{lemma:n-struct}
For $t=1,\ldots,k$, matrix $\mn_t$ consists of $2^{t-1}\times 2^{t-1}$ blocks of Hermitian Toeplitz matrices as the following. The first block-row consists of $2^{t-1}$ Toeplitz matrices
\[
\begin{bmatrix}
\mq_1 & \mq_2 & \cdots & \mq_{2^{t}-1} & \mq_{2^{t-1}}
\end{bmatrix}.
\]
The second block-row is
\[
\begin{bmatrix}
\mq_2 & \mj \mq_1 \mj & \mq_4 & \mj \mq_3 \mj & \cdots & \mq_{2^{t-1}} & \mj \mq_{2^{t-1} - 1} \mj 
\end{bmatrix}
\]
For $s=1,\ldots,t-2$, let $\mp_{s}$ be the matrix consisting of block-rows $1$ to $2^s$, and $\mptil_{s}$ be the matrix consisting of block-rows $2^{s}+1$ to $2^{s+1}$. Let
\[
\mp_s =
\begin{bmatrix}
\mp_{s,1} & \mp_{s,2} & \cdots & \mp_{s,2^{t-s}}
\end{bmatrix}.
\]
Then
\[
\mptil_{s} = \begin{bmatrix}
\mp_{s,2} & \mp_{s,1} & \mp_{s,4} & \mp_{s,3} & \cdots & \mp_{s,2^{t-s}} & \mp_{s,2^{t-s}-1}
\end{bmatrix}.
\]
Moreover for even $j$, $\mq_j=0$. Also the structure of the block-columns of $\mn_t$ is similar to the structure of the block-rows we described above.
\end{lemma}
\begin{proof}
Note that if a real matrix $\mm$ is skew-symmetric, $i\cdot \mm$ is Hermitian. Therefore
the structure of $\mn_1$ follows from Lemma \ref{lemma:key-identity}. For $t=2,\ldots,k$,
note that $\mn_t = i \cdot \imag(\mstil_{t} \mh_{t-1} \mstil_{t}^*)$. Therefore we use the structure of $\mh_{t-1}$ described in Lemma \ref{lemma:h-struct} to prove the structure for $\mn_t$. 
Note that the number of blocks of $\mn_t$ is equal to the number of blocks of $\mh_{t-1}$. Moreover for each block $\mq$ of $\mh_{t-1}$, the corresponding block in $\mn_t$ is $i\cdot \imag(\ms_t \mq \ms_t^*)$, which is a Hermitian Toeplitz matrix by Lemma \ref{lemma:key-identity}. Moreover the structure of block-rows of $\mn_t$ (i.e., the relation between the first block-row and the second block-row and the relation between $\mp_s$ and $\mptil_s$, for $s=1,\ldots,t-2$) follows from the structure of $\mh_{t-1}$ due to Lemma \ref{lemma:h-struct} and the fact that $\mj$ and $\ms_{t}$ commute, i.e., $\mj \ms_{t} = \ms_{t} \mj$. The structure of block columns also follows similarly. 

Finally note that if $\mq$ is a centrosymmetric Hankel matrix, then $\imag(\ms_t \mq \ms_t^*)$ is zero since
\[
\imag(\ms_t \mq \ms_t^*) = \frac{1}{2}(\mq \mj - \mj \mq) = 0,
\]
where the second equality follows from the definition of centrosymmetry (see Definition \ref{def:j}). For the first equality see the proof of Lemma \ref{lemma:key-identity}.
Finally note that for a Hankel matrix $\mq$, by Lemma \ref{lemma:key-identity}, $\real(\ms_t \mq \ms_t^*)$ is a centrosymmetric Hankel matrix. Therefore the top-right and bottom-left blocks of $\real(\ms_t \mq \ms_t^*)$ are also centrosymmetric Hankel matrices. Therefore since the blocks of $\mh_{t-2}$ are Hankel matrices, the blocks $\mq_j$ with even index $j$ in $\mh_{t-1}$ are centrosymmetric Hankel and therefore the corresponding blocks of them in $\mn_t$ are zero.
\end{proof}

Before going further, we need to define the following matrices that allow us to exploit the structures described in Lemmas \ref{lemma:h-struct} and \ref{lemma:n-struct}.

\begin{definition}
\label{def:perm-matrices}
For $t\in[k]$, let
\[
\me_t = \begin{bmatrix}
0_{2^{t-1}\times 2^{t-1}} & 0_{2^{t-1}\times (2^k-2^{t-1})}\\
0_{(2^k-2^{t-1}) \times 2^{t-1}} & \mi_{2^k-2^{t-1}}
\end{bmatrix} \in \F^{2^k\times 2^k},
\]
\[
\mf_t = \begin{bmatrix}
0 & \mi_{2^{t-1}} & 0 & 0 & \cdots & 0 & 0 \\
\mi_{2^{t-1}} & 0 & 0 & 0 & \cdots & 0 & 0 \\
0 & 0 & 0 & \mi_{2^{t-1}} & \cdots & 0 & 0 \\
0 & 0 & \mi_{2^{t-1}} & 0 & \cdots & 0 & 0 \\
\vdots & \vdots & \vdots & \vdots & \ddots & \vdots & \vdots \\
0 & 0 & 0 & 0 & \cdots & 0 & \mi_{2^{t-1}} \\
0 & 0 & 0 & 0 & \cdots & \mi_{2^{t-1}} & 0
\end{bmatrix} \in \F^{2^k\times 2^k}.
\]
\end{definition}
For example, the matrix $\mf_2$ allows us to permute $\mp_1$ in $\mh_k$ to get $\mptil_1$ in Lemma \ref{lemma:h-struct}. We then can remove the first entry of $\mp_1$ using $\me_2$ to prevent it with clashing with the first column. Note that $\me_t^\top = \me_t$ and $\mf_t^\top = \mf_t$. Therefore we can use these matrices for permuting block-columns as well. For the other block-rows/columns, we can use appropriate $\mf_t$'s and $\me_t$'s.  We use these matrices in the proofs of the rest of the section. 

We are now equipped to describe how a representation of symmetric factorization of matrices $\mn_t$, for $t\in[k]$, can be found in linear time.

\begin{theorem}
\label{thm:sym-nt-less-k}
For $t=1,\ldots,k$, we can find $\mx_t$ and $\my_t$ in $\Otil(n)$ time such that  $\mn_t = \mx_t \mx_t^* - \my_t \my_t^*$.
\end{theorem}
\begin{proof}
Let $\begin{bmatrix}
\mq_1 & \cdots & \mq_{2^{t-1}}
\end{bmatrix}$ be the blocks of the first block-row and the first block-column of $\mn_t$, respectively. For $s\in[t]$,
let $\mm_s$ be a matrix with block structure as $\mn_t$ such that all of its blocks are zero except the block-rows/columns $1,\ldots,2^{s-1}$ and its block-rows/columns $1,\ldots,2^{s-1}$ are equal to the block-rows/columns $1,\ldots,2^{s-1}$ of $\mn_t$. For example $\mm_{t} = \mn_{t}$ and 
\[
\mm_1 = \begin{bmatrix}
\mq_1 & \mq_2 & \mq_3 & \cdots & \mq_{2^{t-1}} \\
\mq_2 & 0 & 0 & \cdots & 0 \\
\mq_3 & 0 & 0 & \cdots & 0 \\
\vdots & \vdots & \vdots & \ddots & \vdots \\
\mq_{2^{t-1}} & 0 & 0 & \cdots & 0
\end{bmatrix}. 
\]

For $s\in[2^{t-1}]$, let $\mq_{s,U}$ be an upper triangular matrix that is equal to $\mq_s$ on the diagonal and above the diagonal, and let $\mq_{s,L}$ be a lower triangular matrix with zero diagonal that is equal to $\mq_s$ below the diagonal. Therefore $\mq_s = \mq_{s,U} + \mq_{s,L}$. Moreover since $\mq_s$ is Hermitian by Lemma \ref{lemma:n-struct}, $\mq_s = \mq_{s,U}^* + \mq_{s,L}^*$.

Let $\matu$ and $\ml$ be matrices with the same block structure as $\mn_t$. Moreover suppose all of blocks of $\matu$ and $\ml$ are zero except the first block-row and the first block-column. Let the first block-row and the first block-column of $\matu$ be
\[
\begin{bmatrix}
\frac{1}{2} \mq_1 & \mq_{2,U} & \mq_{3,U} \cdots \mq_{2^{t-1},U}
\end{bmatrix}, \text{ and } \begin{bmatrix}
\frac{1}{2} \mq_1 & \mq_{2,U} & \mq_{3,U} \cdots \mq_{2^{t-1},U}
\end{bmatrix}^*,
\]
respectively. Also let the first block-row and the first block-column of $\ml$ be
\[
\begin{bmatrix}
\frac{1}{2} \mq_1 & \mq_{2,L} & \mq_{3,L} \cdots \mq_{2^{t-1},L}
\end{bmatrix}, \text{ and } \begin{bmatrix}
\frac{1}{2} \mq_1 & \mq_{2,L} & \mq_{3,L} \cdots \mq_{2^{t-1},L}
\end{bmatrix}^*,
\]
respectively. Therefore $\mm_1 = \matu + \ml$. We now give symmetric factorizations for $\matu$ and $\ml$. By construction and Lemma \ref{lemma:key-identity}, the diagonal of $\mq_1$ is zero. Therefore by Lemma \ref{lemma:rank2-symmetric}, the matrix consisting of only the first row and the first column of $\matu$ can be written as $\vv_1 \vv_1^* - \vv_2 \vv_2^*$. Moreover, the matrix consisting of only the first row and the first column of $\matu$ is equal to
\[
\matu - \mDeltatil \matu \mDeltatil^\top,
\]
where 
\[
\mDeltatil = \begin{bmatrix}
\mDelta_{2^{k-t+1}} & 0 & \cdots & 0 \\
0 & \mDelta_{2^{k-t+1}} & \cdots & 0 \\
\vdots & \vdots & \ddots & \vdots \\
0 & 0 & \cdots & \mDelta_{2^{k-t+1}}
\end{bmatrix} \in \F^{2^k\times 2^k}.
\]
Therefore $\matu - \mDeltatil \matu \mDeltatil^\top = \vv_1 \vv_1^* - \vv_2 \vv_2^*$ for some vectors $\vv_1$ and $\vv_2$ that can be computed in $O(n)$ time. Now since $\mq_1$ is Toeplitz and $\mq_{s,U}$'s are upper triangular and Toeplitz, we have
\[
\matu = \sum_{j = 0}^{2^{k-t+1}-1} \mDeltatil^j (\matu - \mDeltatil \matu \mDeltatil^\top) (\mDeltatil^j)^\top.
\]
Therefore setting
\[
\mv_1 = \begin{bmatrix}
\vv_1 & \mDeltatil \vv_1 & \mDeltatil^2 \vv_1 & \cdots & \mDeltatil^{2^{k-t+1}-1} \vv_1
\end{bmatrix},
\]
\[
\mv_2 = \begin{bmatrix}
\vv_2 & \mDeltatil \vv_2 & \mDeltatil^2 \vv_2 & \cdots & \mDeltatil^{2^{k-t+1}-1} \vv_2
\end{bmatrix},
\]
we have $\matu = \mv_1 \mv_1^* - \mv_2 \mv_2^*$. Now consider the matrix that is equal to zero everywhere except on row $2^{k-t+1}$ and column $2^{k-t+1}$ and on that row and column, it is equal to $\ml$. This matrix is equal to
\[
\ml - \mDeltatil^\top \ml \mDeltatil.
\]
Because the diagonal of $\mq_1$ is equal to zero, by Lemma \ref{lemma:rank2-symmetric}, we can find $\vw_1$ and $\vw_2$ such that $\ml - \mDeltatil^\top \ml \mDeltatil = \vw_1 \vw_1^* - \vw_2 \vw_2^*$ in $O(n)$ time. Now since $\mq_1$ is Toeplitz and $\mq_{s,L}$'s are lower triangular and Toeplitz, we have
\[
\ml = \sum_{j = 0}^{2^{k-t+1}-1} (\mDeltatil^j)^\top (\ml - \mDeltatil^\top \ml \mDeltatil) \mDeltatil^j.
\]
Therefore setting
\[
\mw_1 = \begin{bmatrix}
\vw_1 & \mDeltatil^\top \vw_1 & (\mDeltatil^2)^\top \vw_1 & \cdots & (\mDeltatil^{2^{k-t+1}-1})^\top \vw_1
\end{bmatrix}, \text{ and }
\]
\[
\mw_2 = \begin{bmatrix}
\vw_2 & \mDeltatil^\top \vw_2 & (\mDeltatil^2)^\top \vw_2 & \cdots & (\mDeltatil^{2^{k-t+1}-1})^\top \vw_2
\end{bmatrix},
\]
we have $\ml = \mw_1 \mw_1^* - \mw_2 \mw_2^*$. Hence
\[
\mm_1 = \begin{bmatrix}
\mv_1 & \mw_1
\end{bmatrix}\begin{bmatrix}
\mv_1 & \mw_1
\end{bmatrix}^* - 
\begin{bmatrix}
\mv_2 & \mw_2
\end{bmatrix}\begin{bmatrix}
\mv_2 & \mw_2
\end{bmatrix}^*.
\]
We now construct other $\mm_s$'s recursively. By using matrices in Definition \ref{def:perm-matrices} and the structure of $\mn_t$ described in Lemma \ref{lemma:n-struct}, for $s=2,\ldots,t$, we have
\[
\mm_s = \mm_{s-1} + \me_{k-t+s+1}\mf_{k-t+s+1}\mm_{s-1} \mf_{k-t+s+1}^\top \me_{k-t+s+1}^\top
\]
Therefore if $\mp_s$ and $\mg_s$ are matrices such that $\mm_{s} = \mp_s \mp_s^* - \mg_s \mg_s^*$, then for $s=2,\ldots,t$
\[
\mp_s = \begin{bmatrix}
\mp_{s-1} & \me_{k-t+s+1}\mf_{k-t+s+1}\mp_{s-1}
\end{bmatrix},
\]
\[ 
\mg_s = \begin{bmatrix}
\mg_{s-1} & \me_{k-t+s+1}\mf_{k-t+s+1}\mg_{s-1}.
\end{bmatrix}.
\]
This completes the proof since $\mm_t = \mn_t$.
\end{proof}

\noindent
We now prove the main theorem. In addition to Theorem \ref{thm:sym-nt-less-k}, this only requires describing how to find a representation of symmetric factorization of $\mn_{k+1}=\mh_k$ in linear time.
\begin{proof}[Proof of Theorem~\ref{thm:sym-hankel}]
By Theorem \ref{thm:sym-nt-less-k}, we can find $\mx_t$ and $\my_t$ such that $\mn_t = \mx_t \mx_t^* - \my_t \my_t^*$, for $t\in[k]$. Therefore we only need to find $\mx_{k+1}$ and $\my_{k+1}$ such that $\mn_{k+1} = \mx_{k+1} \mx_{k+1}^* - \my_{k+1} \my_{k+1}^*$. Note that $\mn_{k+1} = \mh_{k}$.

The matrix $\mh_k$ is $2^k$-by-$2^k$ and according to Lemma \ref{lemma:h-struct}, it consists of $2^k$-by-$2^k$ blocks. Therefore each block of it is only one entry. In this case for a block $\mq$, we have $\mq=\mj \mq \mj$. Therefore all of the entries on the diagonal of $\mh_k$ are the same. Moreover, each row is just a permutation of the first row and similarly each column is a permutation of the first column. Now let $\mhtil_k$ be the matrix that is equal to $\mh_k$ everywhere except on the diagonal and $\mhtil_k$ is zero on the diagonal, i.e., $\mh_k = \mhtil_k + \mh_k(1,1) \cdot \mi$.

Now let $\mm_0$ be a matrix with all of the entries equal to zero except the first row and the first column and its first row and column is equal to the first row and column of $\mhtil_k$. By Lemma \ref{lemma:rank2-symmetric}, we can find vectors $\vv_1$ and $\vv_2$ such that $\mm_0 = \vv_1 \vv_1^* - \vv_2 \vv_2^*$ in $O(n)$ time.

Now for $t\in[k]$, let $\mm_t$ be the matrix that is zero everywhere except on rows/columns $1,\ldots,2^t$ and its rows/columns $1,\ldots,2^t$ are equal to the corresponding rows/columns of $\mhtil_{k}$. Note that $\mm_k = \mhtil_{k}$.
Now because of the structure of $\mh_k$ described by Lemma \ref{lemma:h-struct}, for $t\in[k]$, we have
\begin{align}
\label{eq:recursive-M}
\mm_t = \mm_{t-1} + \me_t\mf_t \mm_{t-1} \mf_t^\top \me_t^\top
\end{align}
Let $\mp_t$ and $\mg_t$ be such that $\mm_{t} = \mp_t \mp_t^* - \mg_t \mg_t^*$. Then by \eqref{eq:recursive-M}, we have
\begin{align*}
\mm_t & = \begin{bmatrix}
\mp_{t-1} & \me_{t-1} \mf_{t-1} \mp_{t-1}
\end{bmatrix}\begin{bmatrix}
\mp_{t-1} & \me_{t-1} \mf_{t-1} \mp_{t-1}
\end{bmatrix}^*
\\ & - 
\begin{bmatrix}
\mg_{t-1} & \me_{t-1} \mf_{t-1} \mg_{t-1}
\end{bmatrix}\begin{bmatrix}
\mg_{t-1} & \me_{t-1} \mf_{t-1} \mg_{t-1}
\end{bmatrix}^*.
\end{align*}
Therefore $\mp_t = \begin{bmatrix}
\mp_{t-1} & \me_{t-1} \mf_{t-1} \mp_{t-1}
\end{bmatrix}$ and $\mg_t=\begin{bmatrix}
\mg_{t-1} & \me_{t-1} \mf_{t-1} \mg_{t-1}
\end{bmatrix}$.
Therefore, we only need to find $\vv_1$ and $\vv_2$ with $\mm_0 = \vv_1 \vv_1^* - \vv_2 \vv_2^*$ to completely describe $\mhtil_k$. Now if $H_{k}(1,1)\geq 0$, we set
\[
\mx_{k+1} = \begin{bmatrix}
\mp_k & \sqrt{\abs{\mh_{k}(1,1)}} \cdot \mi
\end{bmatrix}, \text{ and } \my_{k+1} = \mg_{k},
\]
and we set
\[
\mx_{k+1} = \mp_k, \text{ and } \my_{k+1} = \begin{bmatrix}
\mg_k & \sqrt{\abs{\mh_{k}(1,1)}} \cdot  \mi
\end{bmatrix},
\]
otherwise. Therefore we can also find a symmetric factorization of $\mn_{k+1}$ in $\Otil(n)$ time and this combined with Theorem \ref{thm:sym-nt-less-k} completes the proof.
\end{proof}

\section{Symmetric Factorization of Inverses of Hankel Matrices}
\label{sec:adv-recursive-algo}

In this section, we prove our result for symmetric factorizations of the inverses of Hankel matrices (Theorem \ref{thm:sym-inv-hankel}). We actually prove a more general result: for a given matrix $\mm\in\R^{n\times n}$ with Sylvester-type displacement rank of two and bit complexity $\ell$, we show how to find a representation of the matrices $\mb$ and $\mc$ with $n$ rows, $\Otil(n)$ columns, and bit complexity $\ell$ in time $\Otil(n^{\omega/2} \cdot \ell)$ such that $\norm{\mm - (\mb \mb^*-\mc \mc^*)}_{\fro} < \frac{1}{2^{\ell}}$. Then since the inverse of a Hankel matrix has a Sylvester-type displacement rank of two, this gives an algorithm for the inverse of a Hankel matrix.
A key technique in our algorithm is the following lemma which is similar to Lemma \ref{lemma:key-identity} with the important difference that when we apply the recursion arising from this lemma, the displacement rank of the matrix doubles (instead of staying the same). This then forces us to stop the recursion when the size of the blocks is $\sqrt{n}$.

\begin{lemma}
\label{lemma:general-fold}
Let $\ms = \frac{1}{2}(1+i)\mi+ \frac{1}{2}(1-i) \mj$, where $\mi$ is the identity matrix and $\mj$ is the exchange matrix (see Definition \ref{def:j}). Let $\mm$ be an $n\times n$ real symmetric matrix with Sylvester-type displacement rank of less than or equal to $r$ with respect to $(\mDelta,\mDelta^\top)$ and $(\mDelta^\top,\mDelta)$. Then $\real(\ms \mm \ms^*)$ is bisymmetric and has a Sylvester-type displacement rank of at most $2r$ with respect to $(\mDelta,\mDelta^\top)$ and $(\mDelta^\top,\mDelta)$. Moreover $\imag(\ms \mm \ms^*)$ is persymmetric Hermitian and has a Stein-type displacement rank of at most $2r+2$ with respect to $(\mDelta,\mDelta^\top)$. Also the diagonal entries of $\imag(\ms \mm \ms^*)$ are zero.
\end{lemma}
\begin{proof}
We first write each part of $\ms \mm \ms^*$. We have
\begin{align*}
\ms \mm \ms^* 
& =
\left( \frac{1}{2}(1+i)\mi+ \frac{1}{2}(1-i) \mj \right) \mm \left( \frac{1}{2}(1-i)\mi+ \frac{1}{2}(1+i) \mj \right)
\\ & = 
\frac{1}{2} \left( \mm + \mj \mm \mj \right) + \frac{i}{2} \left( \mm \mj - \mj \mm \right).
\end{align*}
Note that $\mm+\mj \mm \mj$ is symmetric because both $\mm$ and $\mj \mm \mj$ are symmetric. Moreover
\begin{align*}
(\mm+ \mj \mm \mj) \mj = \mm \mj + \mj \mm = \mj(\mj \mm \mj + \mm).
\end{align*}
Therefore $\real(\ms \mm \ms^*) = \frac{1}{2} \left( \mm + \mj \mm \mj \right)$ is bysymmetric. Now we show it has a Sylvester-type displacement rank of at most $2r$ with respect to $(\mDelta,\mDelta^\top)$. We need to show that 
\[
\rank(\mDelta (\mm + \mj \mm \mj) - (\mm + \mj \mm \mj) \mDelta^\top) \leq 2r.
\]
One can easily verify that $\mDelta \mj = \mj \mDelta^\top$. Therefore
\begin{align*}
\mDelta \mj \mm \mj - \mj \mm \mj \mDelta^\top = \mj(\mDelta^\top \mm - \mm \mDelta) \mj.
\end{align*}
Thus since $\mj$ is a full rank matrix,
\begin{align*}
\rank (\mDelta (\mm + \mj \mm \mj) - (\mm + \mj \mm \mj) \mDelta^\top) 
& \leq 
\rank (\mDelta \mm - \mm \mDelta^\top) + \rank (\mDelta \mj \mm \mj - \mj \mm \mj \mDelta^\top)
\\ & =
\rank (\mDelta \mm - \mm \mDelta^\top) + \rank (\mj(\mDelta^\top \mm - \mm \mDelta ) \mj)
\\ & = 
\rank (\mDelta \mm - \mm \mDelta^\top) + \rank (\mDelta^\top \mm - \mm \mDelta)
\\ & \leq
2r.
\end{align*}
We can bound the Sylvester-type displacement rank of $\real(\ms \mm \ms^*)$ with respect to $(\mDelta^\top, \mDelta)$ in a similar way.
Now we turn to the imaginary part. Since $\mm$ is real and symmetric $\mm^*=\mm$. Hence we have
\begin{align*}
(i(\mm \mj-\mj \mm))^* = -i (\mj \mm^* - \mm^* \mj) = -i (\mj \mm - \mm \mj) = i(\mm \mj - \mj \mm).
\end{align*}
Therefore $\imag(\ms \mm \ms^*)$ is Hermitian. Again since $\mm$ is real and symmetric $\mm^\top=\mm$. Therefore
\begin{align*}
i(\mm \mj - \mj \mm) \mj = i( \mm -\mj \mm \mj)=i\mj(\mj \mm-\mm \mj) = i\mj(\mm^\top \mj-\mj \mm^\top)^\top = i \mj (\mm \mj - \mj \mm)^\top.
\end{align*}

Now we need to show that $\rank ( (\mm \mj - \mj \mm ) - \mDelta (\mm \mj - \mj \mm) \mDelta^\top) \leq 2r$.
We have $\mm \mj - \mDelta \mm \mj \mDelta^\top = (\mm- \mDelta \mm\mDelta)J$. Therefore
\begin{align*}
\rank ( \mm \mj - \mDelta \mm \mj \mDelta^\top) 
& = 
\rank (\mm- \mDelta \mm\mDelta).
\end{align*}
Let $\widetilde{\mDelta}$ be a matrix obtained from $\mDelta$ by changing the $(1,n)$ entry from zero to one. For example
\[
\widetilde{\mDelta} = \begin{bmatrix}
    0 & 0 & 0 & 1 \\
    1 & 0 & 0 & 0 \\
    0 & 1 & 0 & 0 \\
    0 & 0 & 1 & 0 \\
\end{bmatrix}.
\]
Therefore $\widetilde{\mDelta}$ is full-rank. Moreover let $\widetilde{\mi}$ be a matrix obtained by $\mi$ by changing the $(1,1)$ entry from one to zero. For example
\[
\widetilde{\mi} = \begin{bmatrix}
    0 & 0 & 0 & 0 \\
    0 & 1 & 0 & 0 \\
    0 & 0 & 1 & 0 \\
    0 & 0 & 0 & 1 \\
\end{bmatrix}.
\]
Then we have $\widetilde{\mDelta}^\top \mDelta = \widetilde{\mi}$. Since $\widetilde{\mDelta}$ is full rank,
\[
\rank (\mm- \mDelta \mm\mDelta) = \rank (\widetilde{\mDelta}^\top\mm- \widetilde{\mi} \mm\mDelta)
\]
Now note that $\widetilde{\mDelta}$ differ from $\mDelta$ only in the first row and similarly $\widetilde{\mi}$ differ from $\mi$ only in the first row. Therefore 
\[
\rank (\widetilde{\mDelta}^\top\mm- \widetilde{\mi} \mm\mDelta) \leq \rank (\mDelta^\top\mm-  \mm\mDelta) + 1.
\]
Then since the Sylvester-type displacement rank of $\mm$ with respect to $(\mDelta^\top,\mDelta)$ is at most $r$, we have
\[
\rank ( \mm \mj - \mDelta \mm \mj \mDelta^\top) \leq r + 1.
\]
In a similar fashion, one can verify that
\begin{align*}
\rank ( \mj \mm - \mDelta \mj \mm \mDelta^\top) \leq r + 1. 
\end{align*}
Therefore
\begin{align*}
\rank ( (\mm \mj - \mj \mm) - \mDelta (\mm \mj - \mj \mm) \mDelta^\top) 
& \leq  
\rank ( \mm \mj - \mDelta \mm \mj \mDelta^\top) + \rank ( \mj \mm - \mDelta \mj \mm \mDelta^\top)  
\\ & \leq 2r + 2.
\end{align*}
Finally note that since $\mm$ is symmetric $\mm_{i, n + 1 - i} = \mm_{n + 1 - i, i}$ for all $i\in [n]$. Therefore since $(\mm \mj)_{i,i} = \mm_{i, n + 1 - i}$ and $(\mj \mm)_{i,i} = \mm_{n + 1 - i, i}$, we have $(\mm \mj - \mj \mm)_{i,i} = 0$. Thus the diagonal entries of $\imag(\ms \mm \ms^*)$ are zero.
\end{proof}

We are now equipped to prove our result for general matrices with small displacement ranks. We state the theorem for the inverse of Hankel matrices, but our algorithm and proof work for these general matrices due to Lemma \ref{lemma:general-fold}.

\RestyleAlgo{algoruled}
\IncMargin{0.15cm}
\begin{algorithm}[t]
\textbf{Input:} Hankel matrix $\mm\in \F^{2^k\times 2^k}$ \\
Set $\mm_0 = \mm$ \\
\For {$t=1,\ldots,\frac{k}{2}$}{
Set $\mn_t = i \cdot \imag(\mstil_{t} \mm_{t-1} \mstil_{t}^*)$ \\
Set $\mm_t = \real(\mstil_{t} \mm_{t-1} \mstil_{t}^*)$ \\
}
Set $\mn_{\frac{k}{2}+1} = \mm_{\frac{k}{2}}$ \\
\For {$t=1,\ldots,\frac{k}{2}+1$}{
Set $\mx_t$ and $\my_t$ to be matrices such that $\mn_t = \mx_t \mx_t^* - \my_t \my_t^*$ \\
Set $\mb_t = \mstil_1^* \mstil_2^* \cdots \mstil_{t-1}^* \mstil_t^* \mx_t$ and $\mc_t = \mstil_1^* \mstil_2^* \cdots \mstil_{t-1}^* \mstil_t^* \my_t$
}
\Return $\begin{bmatrix}
\mb_{\frac{k}{2}+1} & \mb_{\frac{k}{2}} & \cdots & \mb_2 & \mb_1
\end{bmatrix}$ and $\begin{bmatrix}
\mc_{\frac{k}{2}+1} & \mc_{\frac{k}{2}} & \cdots & \mc_2 & \mc_1
\end{bmatrix}$
\caption{Symmetric Factorization of General Matrices with Small Sylvester-Type Displacement Rank}
\label{alg:sylvester-sym-main-algo}
\end{algorithm}

\symmetricInverseHankel*
\begin{proof}
We consider a general matrix $\mm$ that has Sylvester-type displacement rank of two with respect to $(\mDelta, \mDelta^\top)$ and $(\mDelta^\top, \mDelta)$. For example, one can consider $\mm = \mh^{-1}$. Note that we can find a representation of $\mh^{-1}$ as $\mx \my^\top$ in $\Otil(n\cdot \ell)$ time using the approach of \cite{peng2021solving}.
Without loss of generality, we assume that $n=2^k$ and $k$ is even because otherwise, we can extend the matrix to a size of power of four by appropriately copying the entries to make sure Sylvester-type displacement rank does not change. 

We show that Algorithm \ref{alg:sylvester-sym-main-algo} outputs the desired factorization in the specified running time and bit complexity. Note that this algorithm is similar to the one we used for the symmetric factorization of Hankel matrices (Algorithm \ref{alg:hankel-sym-main-algo}). The main difference is that we recurse only for $\frac{k}{2}$ iterations. This is because by Lemma \ref{lemma:general-fold}, the blocks of $\mm_t$ have a Sylvester-type displacement rank of $2^{t+1}$ and a size of $\frac{n}{2^t}\times \frac{n}{2^t}$. So for $t=\frac{k}{2}$, the size of each block is $\sqrt{n} \times \sqrt{n}$ while its displacement rank is $2 \sqrt{n}$. In other words, the displacement rank is more than the rank of the matrix and therefore it does not help with speeding up the computation of symmetric factorization. The rest of the proof basically is similar to the symmetric factorization of Hankel matrices.

Similar to Lemma \ref{lemma:h-struct} and \ref{lemma:n-struct}, $\mn_t$ consists of $2^{t-1}\times 2^{t-1}$ Hermitian blocks. However instead of being Toeplitz, by Lemma \ref{lemma:general-fold} each block has a Stein-type displacement rank of $2^{t+1} +2$ with respect to $(\mDelta,\mDelta^\top)$. Moreover each block is persymmetric. The structure of the blocks is also similar to Lemma \ref{lemma:n-struct}. More specifically, the first block row of $\mn_t$ consists of $2^{t-1}$ matrices
\[
\begin{bmatrix}
\mq_1 & \mq_2 & \cdots & \mq_{2^{t}-1} & \mq_{2^{t-1}}
\end{bmatrix}.
\]
The second block-row is
\[
\begin{bmatrix}
\mq_2 & \mj \mq_1 \mj & \mq_4 & \mj \mq_3 \mj & \cdots & \mq_{2^{t-1}} & \mj \mq_{2^{t-1} - 1} \mj 
\end{bmatrix}
\]
For $s=1,\ldots,t-2$, let $\mp_{s}$ be the matrix consisting of block-rows $1$ to $2^s$, and $\mptil_{s}$ be the matrix consisting of block-rows $2^{s}+1$ to $2^{s+1}$. Let
\[
\mp_s =
\begin{bmatrix}
\mp_{s,1} & \mp_{s,2} & \cdots & \mp_{s,2^{t-s}}
\end{bmatrix}.
\]
Then
\[
\mptil_{s} = \begin{bmatrix}
\mp_{s,2} & \mp_{s,1} & \mp_{s,4} & \mp_{s,3} & \cdots & \mp_{s,2^{t-s}} & \mp_{s,2^{t-s}-1}
\end{bmatrix}.
\]
The structure of block columns of $\mn_t$ is also similar to the above. For example, the matrix consisting only of the first row and the first column of $\mn_t$ is like the following.
\[
\mz_t = \begin{bmatrix}
\mq_1 & \mq_2 & \mq_3 & \cdots & \mq_{2^{t-1}} \\
\mq_2 & 0 & 0 & \cdots & 0 \\
\mq_3 & 0 & 0 & \cdots & 0 \\
\vdots & \vdots & \vdots & \ddots & \vdots \\
\mq_{2^{t-1}} & 0 & 0 & \cdots & 0
\end{bmatrix}. 
\]
All the other block-rows and block-columns are obtained by taking permutations of the first block-row and the first block-column, respectively. Moreover the block-row and the block-column with the same index are obtained by the same permutation. Since $\mq_1$ is a Hermitian matrix, $\mq_1-\mDelta \mq_1 \mDelta^\top$ is also Hermitian. Therefore the eigenvalues of $\mq_1-\mDelta \mq_1 \mDelta^\top$ are real and since it is an $\frac{n}{2^t} \times \frac{n}{2^{t}}$ of rank $2^{t+1}+2$, we can find its eigenvalue decomposition in $\Otil(\frac{n}{2^t}  \cdot 2^{t \cdot (\omega-1)} \cdot \ell)$. Then we can separate the eigenvectors associated with positive and negative eigenvalues into matrices $\mdtil$ and $\metil$, respectively, and write $\mq_1-\mDelta \mq_1 \mDelta^\top$ as $\mdtil \mdtil^* - \metil \metil^*$. Then by writing 
\[
\mq_1 = \mq_1-\mDelta \mq_1 \mDelta^\top + \mDelta(\mq_1-\mDelta \mq_1 \mDelta^\top) \mDelta^\top + \mDelta \mDelta (\mq_1-\mDelta \mq_1 \mDelta^\top) \mDelta^\top \mDelta^\top + \cdots,
\]
we can take the appropriate shifts and write $\mq = \md \md^* - \mc \mc^*$, where
\[
\md = \begin{bmatrix}
    \mdtil & \mDelta \mdtil & \mDelta \mDelta \mdtil & \cdots
\end{bmatrix}, ~~ \text{and} ~~
\me = \begin{bmatrix}
    \metil & \mDelta \metil & \mDelta \mDelta \metil & \cdots
\end{bmatrix}.
\]
Then we can write $\mz_t$ as $\mf_1 \mf_1^* - \mg_1 \mg_1^*$, where
\[
\mf_1 = 
\begin{bmatrix}
    \md & \mi \\
    0 & \mq_2 \\
    0 & \mq_3 \\
    \vdots & \vdots \\
    0 & \mq_{2^{t-1}} \\
\end{bmatrix}, ~~ \text{and} ~~
\mg_1 = 
\begin{bmatrix}
    \me & \mi & 0 \\
    0 & 0 & \mq_2 \\
    0 & 0 & \mq_3 \\
    \vdots & \vdots \\
    0 & 0 & \mq_{2^{t-1}} \\
\end{bmatrix}
\]
Then since the rest of block-rows and block-columns of $\mn_t$ are obtained by permuting its first block-row and block-column, we can obtain the factorization by taking appropriate shift and permutations of $\mf_1$ and $\mg_1$. This then produces $\mbtil_t$ and $\mctil_t$ such that $\mn_t = \mbtil_t \mbtil_t^* - \mctil_t \mctil_t^*$. Then by defining $\mb_t = \mstil_1^* \mstil_2^* \cdots \mstil_{t-1}^* \mstil_t^* \mbtil_t$ and $\mc_t = \mstil_1^* \mstil_2^* \cdots \mstil_{t-1}^* \mstil_t^* \mctil_t$, and 
$\mb = \begin{bmatrix}
\mb_{\frac{k}{2}+1} & \mb_{\frac{k}{2}} & \cdots & \mb_2 & \mb_1
\end{bmatrix}$ and $\mc = \begin{bmatrix}
\mc_{\frac{k}{2}+1} & \mc_{\frac{k}{2}} & \cdots & \mc_2 & \mc_1
\end{bmatrix}$, we have
\[
\mm = \mb \mb^* - \mc \mc^*.
\]
Note that for each $\mn_t$, we only need to compute an eigendecomposition for its corresponding $\mq_1 - \mDelta \mq_1 \mDelta^\top$, and the rest of the decomposition for $\mn_t$ follows from deterministic shifts and permutations. We do not even need to compute $\mq_2,\cdots,\mq_{2^{t-1}}$ since the entries of them can be obtained by $O(\log n)$ addition/subtraction of the entries of the original matrix. Since $\mq_1$ is obtained by taking summations over principal submatrices of the original matrix, its bit complexity and operator norm are the same as the original matrix (up to a $\log n$ factor). Therefore the cost of computing this eigendecomposition is 
\[
\Otil(\frac{n}{2^t}  \cdot 2^{t \cdot (\omega-1)} \cdot \ell).
\]
because $\mq_1-\mDelta \mq_1 \mDelta^\top$ is an $\frac{n}{2^t} \times \frac{n}{2^{t}}$ of rank $2^{t+1}+2$, and we can add a random matrix with small entries (of less than $\frac{1}{n^2 \ell}$). Adding this random matrix, does not cause us to go above the error threshold but it causes a gap between the eigenvalues of $\mq_1-\mDelta \mq_1 \mDelta^\top$ and causes it to have $\poly{n\cdot 2^\ell}$ condition number which results in the running time stated above --- see \cite{banks2022pseudospectral}.
Therefore the total cost of computing the representation is
\[
\sum_{t=1}^{1+k/2}\Otil(\frac{n}{2^t}  \cdot 2^{t \cdot (\omega-1)}\cdot \ell) = \sum_{t=1}^{1+k/2}\Otil(n  \cdot 2^{t \cdot (\omega-2)}\cdot \ell).
\]
Since $t\leq \frac{k}{2}+1$ and $n^{1+k/2}=2\sqrt{n}$, this is bounded by $\Otil(n^{\omega/2}\cdot \ell)$.
\end{proof}

\section{Discussion and Conclusion}
\label{sec:conclusion}
In this paper, we presented novel super-fast algorithms to find a representation of symmetric factorizations of the form $\mb \mb^* - \mc \mc^*$ for Hankel matrices and their inverses. Our running times for Hankel matrices and their inverses are $\Otil(n \cdot \ell)$ and $\Otil(n^{\omega/2} \cdot \ell)$. We also conjectured that it is possible to find factorizations of the form $\mb \mb^*$ for these problems in the same running times. We explained how our conjectures lead to faster algorithms for solving a batch of linear systems faster than the approach of \cite{peng2021solving} and how they lead to a faster-than-matrix-multiplication algorithm for solving sparse poly-conditioned linear programs.

Here we present a statement that has the same implications. This is weaker than our conjectures but stronger than the results we proved. Suppose we find $\mb$ and $\mc$ such that $\mb \mb^* - \mc \mc^* = (\mk^\top \ma \mk)^{-1}$ and $\mb \mb^* \preceq \text{poly}(n) (\mk^\top \ma \mk)^{-1}$, then since $\mb \mb^* - \mc \mc^* = (\mk^\top \ma \mk)^{-1} \succeq 0$, $\mc \mc^* \preceq \text{poly}(n) (\mk^\top \ma \mk)^{-1}$. Then by the argument of Equations \eqref{eq:sym-argument}, $\mk \mb$ and $\mk \mc$, both will have low bit-complexity since 
\[\mk \mb \mb^* \mk^\top \preceq \text{poly}(n) \mk (\mk^\top \ma \mk)^{-1} \mk^\top = \text{poly}(n) \ma^{-1}\] 
and 
\[\mk \mc \mc^* \mk^\top \preceq \text{poly}(n) \mk (\mk^\top \ma \mk)^{-1} \mk^\top = \text{poly}(n) \ma^{-1}. \]
Then using the linear operator $\mbtil \mbtil^* - \mctil \mctil^*$ for $\mbtil=\mk \mb$ and $\mctil = \mk \mc$ leads to the same running times presented in Equations \eqref{eq:our-running-time} and \eqref{eq:our-inv-maintenance-run-time}.

\section*{Acknowledgement}

We express our gratitude to Richard Peng for providing valuable comments and engaging discussions that greatly influenced the development of this paper. Furthermore, we extend our appreciation to Santosh Vempala for insightful discussions regarding sparse linear systems. Their valuable input and expertise have played a pivotal role in inspiring the author to produce this work.

\bibliographystyle{alpha}
\bibliography{ref}

\clearpage
\appendix
\section{Omitted Proofs}
\label{sec:app-proofs}
\applyK*
\begin{proof}
We first discuss he running time of computing $\mk \mb$.
We partition the rows of $\mb$ to $m$ blocks $\mb_1,\cdots, \mb_m\in\R^{s\times n}$ as the following
\[
\mb = \begin{bmatrix}
    \mb_1 \\ \mb_2 \\ \vdots \\ \mb_m
\end{bmatrix}.
\]
Then we have
\[
\mk \mb = \sum_{i=1}^m \ma^{i-1} \mg \mb_i.
\]
For $j\in[m]$, define
\[
\mm_{j} = \sum_{i=j}^m \ma^{i-j} \mg \mb_i.
\]
Therefore $\mk \mb = \mm_1$ and for $j \geq 2$,
\begin{align}
\label{eq:kb-recursion}
\mm_{j-1} 
& = 
\sum_{i=j-1}^m \ma^{i-j+1} \mg \mb_i 
\\ & =\nonumber
\mg \mb_{j-1} + \sum_{i=j}^m \ma^{i-j+1} \mg \mb_i
\\ & =\nonumber
\mg \mb_{j-1} + \ma \sum_{i=j}^m \ma^{i-j} \mg \mb_i
\\ & = \nonumber
\mg \mb_{j-1} + \ma \mm_j.
\end{align}
Therefore given $\mm_j$, we can compute $\mm_{j-1}$ in time $\Otil(mnr \cdot \ell + \nnz{\ma} \cdot r \cdot (\ell+\ell_j))$, where $\ell_j$ is the bit complexity of $\mm_j$. The first term in the running time is for computing $\mg \mb_{j-1}$ and essentially follows from the fact that each column of $\mb_{j-1}$ can be multiplied by $\mg$ in $\Otil(m n \cdot \ell)$ time because the bit complexity of $\mb$ is at most $m \ell$ and $\mg$ has $\Otil(n)$ nonzero entries.
The second term in the running time is for computing $\ma \mm_j$ and follows from a similar argument. Moreover $\ell_{j-1} = \Otil(\max\{m\cdot \ell,\ell+\ell_j\})$.

Since $\mm_m=\mg \mb_{m}$, it can be computed in $\Otil(mn r \cdot \ell)$ time and it has a bit complexity of $\Otil(m\cdot \ell)$. Therefore by the recurrence relation \eqref{eq:kb-recursion}, $\mk\mb = \mm_1$ can be computed in time 
\[
\Otil(\nnz{\ma} \cdot r \cdot m^2 \cdot \ell)
\]
because the recursion only goes for $m$ steps and the bit complexity of the intermediate matrices stays $\Otil(m\cdot \ell)$.

Computing $\mk^\top \mb$ is less complicated. Note that 
\[
\mk^\top \mb = \begin{bmatrix}
    \mg^\top \mb \\
    \mg^\top \ma^\top \mb \\
    \mg^\top (\ma^\top)^2 \mb \\
    \vdots \\
    \mg^\top (\ma^\top)^{m-1} \mb
\end{bmatrix}.
\]
Given $(\ma^\top)^j\mb$, $\mg^\top (\ma^\top)^j \mb$ and $(\ma^\top)^{j+1}\mb$ can be computed in time $\Otil(n\cdot r \cdot (\ell+\ell_j))$ and $\Otil(\nnz{\ma} \cdot r \cdot (\ell + \ell_j))$, where $\ell_j$ is the bit complexity of $(\ma^\top)^j\mb$. Therefore this gives a total running time of 
\[
\Otil(\nnz{\ma} \cdot r \cdot m^2 \cdot \ell).
\]
\end{proof}

\end{document}